\documentclass[11pt,a4paper]{amsart}
\pdfoutput=1

\usepackage[german,english]{babel}
\usepackage{amssymb}
\usepackage{mathrsfs} % needed for the curly 'O'
\usepackage{amsmath}
\usepackage{pdfpages}

\theoremstyle{plain}
\newtheorem{theorem}{Theorem}

\theoremstyle{definition}

\newtheorem{example}[theorem]{Example}
\newtheorem{remark}[theorem]{Remark}

\newcommand{\dd}{\mathrm{d}}
\newcommand{\cotanh}{\mathrm{cotanh}}
\newcommand{\FF}{\mathbf{F}} %'full-affine'
\newcommand{\PP}{\mathbf{P}} %'projective'
\newcommand{\EE}{\mathbf{E}} %'Euclidean'
\newcommand{\Order}{\mathscr{O}}
\newcommand{\threeasterisks}{\begin{center}$\ast\quad \ast$\\ $\ast$\end{center}}

\title[Curvature Functionals for Curves in the Equi-Affine Plane]{Curvature Functionals\\ \rule{0pt}{15pt}{\tiny for}\\ \rule{0pt}{17pt}Curves in the Equi-Affine Plane.}

\author{Steven Verpoort}

\dedicatory{(K.U.Leuven, Belgium, and  Masaryk University\,/\,Eduard \v{C}ech Center, Czech Republic.)}

\thanks{
The author, who was employed at K.U.Leuven during the commencement of this work, while he was supported by Masaryk University (Brno) during its conclusion, is thankful to both these institutions. 
This research was partially supported by 
the Research Foundation Flanders (project G.0432.07) and the Eduard \v{C}ech Center for Algebra and Geometry (Basic Research Center no. LC505).
}

\begin{document}

\maketitle
%%%%%%%%%%%%%%%

%%%%%%%%%%%%%%%%%%%%%%%%%%%%%%%%%%%%%%%%%%%%
\begin{abstract}
After having given the general variational formula for the functionals indicated in the title, the critical points of the integral of the equi-affine curvature under area constraint and the critical points of the full-affine arc-length are studied in greater detail.
 
\end{abstract}
%%%%%%%%%%%%%%%%%%%%%%%%%%%%%%%%%%%%%%%%%%%%

\section{Preface.\!\!${}^1$}
%%%%%%%%%%%%%%%

\hspace{2\textwidth}\footnote{\selectlanguage{german}\textit{"'In Vorworten pflegen die Verfasser (wie Zahn\"arzte) die L\"ucken zu bohren, die sie 
sp\"ater stopfen wollen."`}\\ ${\,}^{\,}$ \hfill ---W. Blaschke \cite{blaschke_kin}. \selectlanguage{english}}

\selectlanguage{english}
One of the many striking features of W. Blaschke's landmark book ``\textit{Vorlesungen II}'' \cite{blaschke_vorlesungenII}, being the first treatise on equi-affine differential geometry which also at present day remains in multiple aspects the best introduction to the subject, is the close analogy between the development of the main body of the equi-affine theory and the exposition of classical differential geometry  \cite{blaschke_vorlesungenI}.
Although Blaschke showed a great interest in isoperimetric and variational problems, a rare topic which breaks this similarity is precisely the question of the infinitesimal change of a planar curvature functional, for Radon's problem is indeed covered with some detail in \cite{blaschke_vorlesungenI} whereas in \cite{blaschke_vorlesungenII} only the variation of the equi-affine arc-length is considered. 
In fact, even after having asked many colleagues, to whom I extend my gratitude, I have not been able to find \textit{any} article on equi-affine curvature functionals for planar curves (although centro-affine curvature functionals for planar curves are treated in \cite{huang} whereas \cite{mihailescu_2} covers a variational problem w.r.t.\ the full-affine group).

For this reason, I propose to study some basic variational problems for curves in the equi-affine plane.

%%%%%%%%%%%%%%%%%%%%%%%%%%%%%%%%%%%%%%%%%%%%%%%%%%%%%%%%%%%%%%%%%%%%%%%%%%%
%%%%%%%%%%%%%%%%%%%%%%%%%%%%%%%%%%%%%%%%%%%%%%%%%%%%%%%%%%%%%%%%%%%%%%%%%%%

\section{The General Variational Formula.}
%%%%%%%%%%%%%%%

Let $\gamma$ be a curve in the equi-affine plane $\mathbb{A}^2$ (which is equipped with a fixed area form $\left|\cdot,\cdot\right|$). It will be assumed that $\gamma$ is parametrised by \textit{equi-affine arc-length} $s$, which means that $\left|\gamma',\gamma''\right|=1$. 
A prime will always stand for a derivative w.r.t.\ the equi-affine arc-length parameter, \textit{i.e.}, we adopt the shorthand notation $\mu' = \frac{\partial \phantom{s}}{\partial s}\mu$ for any function $\mu$.
We will denote $\{T,\,N\}$ for the equi-affine Frenet frame consisting of the tangent vector  $T=\gamma'$ and the Blaschke normal $N=\gamma''$. The \textit{equi-affine curvature} $\kappa$ is defined by means of the equation
\begin{equation}
\label{eq:defkappa}
N'=-\kappa\,T
\qquad\qquad
\textrm{(or, equivalently, $\gamma'''=-\kappa\,\gamma'$\,).}
\end{equation}

In the course of this article, some remarks concerning invariants of curves with respect to the Euclidean, the full-affine or the projective group will be made. These invariants will bear an index $\EE$, $\FF$ or $\PP$. 

The position vector field $P$ (with respect to an arbitrarily chosen origin) can be written as $P=-\rho\,N+\phi\,T$, where the function $\rho$ is called the \textit{equi-affine support function}. By expressing the fact that $T=P'$ one finds two first-order differential equations for $\rho$ and $\phi$, which can be combined to give
\begin{equation}
\label{eq:rho_ode}
\rho'' + \kappa\,\rho = 1\,.
\end{equation}

It can easily be seen that for a deformation of a curve $\gamma$ with deformation vector field $f\,N+g\,T$, the equi-affine tangent of the deformed curve $\gamma_t$ is given by
\[
T_t(s) = T(s) - \frac{t}{3}\,\big( f''(s)+\kappa(s)\, f(s) \big)\,T (s) + t\,\big(\cdots\big)\,N(s) + \Order(t^2).
\]
Because the equi-affine curvature of the deformed curve $\gamma_t$ is precisely the centro-affine curvature of its equi-affine tangent image $T_t$, the variational formula for equi-affine curvature functionals can immediately be found from eq.\ (2.6) of \cite{huang}.
It is not necessary to calculate the component along $N$ of the deformation vector field of the tangent image (which is indicated by dots in the above formula), because it merely represents an infinitesimal reparametrisation of the tangent image.
Furthermore, it will be assumed that the deformation vector field is compactly supported such that no boundary terms occur. This requirement is automatically satisfied for closed curves. In this way the following variational formula is obtained, in which both $f$ and $\kappa$ should be evaluated in $s$:
\begin{equation}
\label{eq:delta_functional}
\delta \int F(\kappa)\,\dd s =
\frac{1}{3}\int f \left[\left\lgroup
\frac{\partial^2\ }{\partial s^2} + \kappa
\right\rgroup\!\!\left\lgroup
F'''(\kappa)\,(\kappa')^2
+ F''(\kappa)\,\kappa''
+4\,F'(\kappa)\,\kappa
-2\,F(\kappa)
\right\rgroup\right ]\,\dd s\,.
\end{equation}
For every curve $\gamma=(x,y)$, the general solution of the differential equation $\left\lgroup\frac{\partial^2\ }{\partial s^2} + \kappa\right\rgroup\xi=0$ is given by $\xi=A\,x'+B\,y'$. Therefore the Euler--Lagrange equation can immediately be integrated twice, with the following result:

\begin{theorem}
\label{thm:intFkappa}
A curve $\gamma=(x,y)$ in the affine plane is a critical point of the equi-affine curvature functional $\int F(\kappa)\,\dd s$ if and only if
\[
F'''(\kappa)\,(\kappa')^2
+ F''(\kappa)\,\kappa''
+4\,F'(\kappa)\,\kappa
-2\,F(\kappa)
=
A\,x'+B\,y'
\]
for some $A$, $B$ in $\mathbb{R}$.
\end{theorem}

We will now consider two special instances of the above variational formula, namely, the functionals 
$\int \kappa\,\dd s$  (\S\,\ref{sec:totcurv}) and $\int \sqrt{\kappa}\,\dd s$ (\S\,\ref{sec:sqrtcurv}).

%%%%%%%%%%%%%%%%%%%%%%%%%%%%%%%%%%%%%%%%%%%%%%%%%%%%%%%%%%%%%%%%%%%%%%%%%%%
%%%%%%%%%%%%%%%%%%%%%%%%%%%%%%%%%%%%%%%%%%%%%%%%%%%%%%%%%%%%%%%%%%%%%%%%%%%

\section{The Total Equi-Affine Curvature.}
%%%%%%%%%%%%%%%
\label{sec:totcurv}

Unlike in Euclidean differential geometry, the total equi-affine curvature $\int \kappa\,\dd s$ of a closed curve is not topologically 
%%begin footnote
invariant,\!\!\footnote{In fact, not any non-zero functional of the form $\int F(\kappa)\,\dd s$ is a topological invariant, as can be deduced from the fact that for no non-zero function $F$ the right-hand side of (\ref{eq:delta_functional}) vanishes identically.} 
%%end footnote
which makes the variational problem worthy of study. 
The variational formula
\begin{equation}
\label{eq:delta_int_kappa}
\delta \int \kappa\,\dd s \,=\, \frac{2}{3} \int f\, \big(\kappa''+\kappa^2\big)\,\dd s\,,
\end{equation}
can be derived as well from formula (4.1.12) of \cite{li_etal}, p.\ 213 (see also \cite{li1988}).

\begin{theorem}
\label{thm:intkappa}
A curve $\gamma=(x,y)$ in the affine plane is a critical point of $\int \kappa\,\dd s$ if and only if
\begin{equation}
\label{eq:AxByCrho}
\kappa = A\,x'+B\,y'
\end{equation}
for appropriate constants $A$, and $B$, \textit{i.e.}, the centro-affine curvature of the tangent image described by $T$ is a linear function of its position vector field.
\end{theorem}

\begin{theorem}
\label{thm:intkappaarea}
A curve $\gamma=(x,y)$ in the affine plane is a critical point of $\int \kappa\,\dd s$ under area constraint, without being an unconstrained critical point of the total equi-affine curvature, if and only if the origin can be chosen in such a way that the equi-affine support function becomes a non-zero multiple of the equi-affine curvature.
\end{theorem}

\begin{proof}
We recall that, under a deformation with deformation vector field $f\,N+g\,T$, the area bounded by the curve changes according to $\delta\,\textrm{Area}=-\int f \, \dd s$. Consequently, the Euler--Lagrange equation which is satisfied by the curve reads
\begin{equation}
\label{eq:odekappa}
\kappa''+ \kappa^2 = C
\end{equation}
for some non-zero constant $C$. Now consider the vector field 
\[
M=\kappa\,N - \kappa'\,T
\]
along the curve. It follows from (\ref{eq:odekappa}) that $M+C\,P$ is a constant vector field, which becomes zero after a suitable translation has been applied. We then have 
$\kappa=-\left|M,T\right|=C\left|P,T\right|=C\,\rho$.
\end{proof}

The Theorem~\ref{thm:intkappaarealength} below, and perhaps Theorems~\ref{thm:intFkappa},~\ref{thm:intkappa},~\ref{thm:intkappaarea} and \ref{thm:intsqrt} of this article as well, are reminiscent to the known result that a curve in $\mathbb{E}^2$ is a critical point of the Bernoulli--Euler bending energy $\int  (\kappa_{\mathbf{E}})^2\,\dd s_{\mathbf{E}}$ under constrained area and arc-length  (without being a critical point of this variational problem if the area constraint is relaxed) if and only if the Euclidean curvature can be written as $\kappa_{\mathbf{E}}=A+B\,\|\gamma\|^2$ for some choice of origin and some $A,B$ with $B\neq 0$ (\cite{arr_capo_chry_guven}, eq.\ (38)).

\begin{theorem}
\label{thm:intkappaarealength}
A curve $\gamma=(x,y)$ in the affine plane is a critical point of $\int \kappa\,\dd s$ under constrained area and total equi-affine length (without being a critical point of $\int \kappa\,\dd s$ under arc-length constraint only) if and only there holds $\kappa = C + A\,\rho$ for some $A,C$ (with $C\neq 0$) and some choice of origin.
\end{theorem}
\begin{proof}
This follows similarly from the equations $\kappa''+\kappa^2=C+A\,\kappa$ (with $C\neq 0$) and $M=\kappa\,N-\kappa'T-A\,N$.
\end{proof}

\begin{theorem}
The ellipses are the only closed curves which are a critical point of the functional $\int\kappa\,\dd s$ with respect to deformations under which the equi-affine arc-length is preserved.
\end{theorem}
\begin{proof}
Let $\gamma$ be a closed curve which solves the variational problem as stated in the theorem. From the Euler--Lagrange equation
\begin{equation}
\label{eq:EL}
\kappa''+\kappa^2 = A\,\kappa \qquad (A\in\mathbb{R})
\end{equation}
can be inferred that
\[
\int \kappa^2\,\dd s 
= \int \big(\kappa''+\kappa^2\big)\,\dd s = A \int \kappa\,\dd s\,.
\]
From (\ref{eq:rho_ode}) we deduce that 
\[
\int \kappa\,\dd s 
= \int \kappa \big(\rho''+\kappa\,\rho\big)\,\dd s
= \int \big(\kappa''+\kappa^2\big)\,\rho\,\dd s 
= A \int \rho\,\kappa\,\dd s
= A \int \dd s\,.
\]
A combination of the previous two equations results in
\[
\int (\kappa-A)^2\,\dd s
=
\int \kappa^2\,\dd s -2\,A \int \kappa\,\dd s + \int A^2\,\dd s
=
(A^2-2A^2+A^2)\int\dd s 
=
0\,,
\]
which means that $\kappa$ is equal to the constant $A$.
\end{proof}

\begin{remark}
It has been shown in \cite{blaschke_vorlesungenII}, \S\,27.24 that every simple closed curve in the affine plane satisfies 
\[
\int \kappa\,\dd s \,\sqrt[3]{\mathrm{Area}}\leqslant 2\,\pi^{\frac{4}{3}}\,,
\] 
and equality occurs precisely for the ellipses. By combining this equation with the Blaschke inequality, it follows that 
\begin{equation}
\label{eq:intkappaints}
\int\kappa\,\dd s  \int \dd s \leqslant  4\,\pi^{2}\,,
\end{equation}
in which equality is also characteristic to the ellipses.
\end{remark}

\threeasterisks

We now intend to give a full classification of all curves which are a critical point of $\int\kappa\,\dd s$ under area constraint. Our first observation is that the Euler--Lagrange equation (\ref{eq:odekappa}) can be integrated. Indeed, if the notation $f=\frac{-1}{6}\kappa$, $g_2=\frac{C}{3}$ is adopted, whereas $g_3$ stands for an integration constant, the following requirement on
the function $f$ and the constants $g_2$, $g_3\in\mathbb{R}$ is found:
\begin{equation}
\label{eq:odekappabis}
(f')^2 = 4\,f^3-g_2\,f-g_3
\qquad
\textrm{(or equivalently, $(\kappa')^2=-\frac{2}{3}\,\kappa^3+6\,g_2\,\kappa-36\,g_3$)\,.}  
\end{equation}

The constant $g_2$ vanishes if and only if the curve is a critical point of the total equi-affine curvature without constraint. In the assumption $g_2\neq 0$, the constant $g_3$ vanishes if and only if the curve traced out by the vector $\beta = \big({1}\big/{\sqrt{|\kappa|}}\big)P$ (the position vector field $P$ being chosen as in Theorem~\ref{thm:intkappaarea}) is a straight line. This last fact can be immediately found from the formula
$\beta''=-27\,\frac{g_3}{\kappa^2}\,\beta$.

Furthermore, it may be instructive to have in mind the picture of the phase plane $(\kappa,\kappa')$, where the equation (\ref{eq:odekappabis}) determines a cubic $\mathcal{C}$ which is symmetric w.r.t.\ the horizontal axis which it intersects in precisely three distinct points if the cubic discriminant $\Delta = (g_2)^3-27\,(g_3)^2$ is strictly positive, and in precisely one point if this discriminant is strictly negative (see Figure~\ref{fig:cubicC}). While a point in the affine plane describes a curve $\gamma$ which solves the variational problem, the corresponding point in the phase plane will describe a part of this cubic $\mathcal{C}$. 

\noindent\framebox{\textsc{Generic Case}. $g_2\neq 0$ and $\Delta\neq 0$.}\raisebox{-9pt}{\rule{0pt}{24pt}}\ \ 
Let us first consider this generic case, which will be split into Cases (A), (B) and (C) afterwards.
It is known that complex constants $\omega_1$ and $\omega_2$ can be found, which are not proportional over $\mathbb{R}$, for which
\[
g_2 =  \sum \frac{60}{(2\,m_1\,\omega_1+2\,m_2\,\omega_2)^4}
\quad
\textrm{and}
\quad
g_3 = \sum \frac{140}{(2\,m_1\,\omega_1+2\,m_2\,\omega_2)^6}\,,
\] 
the summations being taken over $m_1,m_2\in\mathbb{Z}^2\setminus\{(0,0)\}$ (see \cite{whittaker_watson}, \S\,20.22; 21.73). Moreover if $\wp$ denotes the Weierstrass elliptic function with half-periods $\omega_1$ and $\omega_2$ and invariants $g_2$, $g_3$, then the only complex functions $f$ which solve equation (\ref{eq:odekappabis}) are $z\mapsto \wp(z_0\pm z)$ (for $z_0\in\mathbb{C}$).

We remark that the restrictions of $\wp$ to the real and the imaginary line are periodic, and will denote their periods as $2\,\varpi_1$ (a strictly positive number) resp. $2\,\varpi_2$ (a strictly positive multiple of $i$)\footnote{%
%%begin footnote
Compare \cite{ab_st}, Ch. 18, and \cite{whittaker_watson}, p. 444, Ex. 1. 
If $\Delta > 0$, the grid of poles of $\wp$ is rectangular and the standard choice of half-periods is $\omega_1=\varpi_2$ and $\omega_2=\varpi_1+\varpi_2$. If the cubic discriminant $\Delta$ is strictly negative, this grid is rhombic and the standard half-periods are given by $\omega_1=\frac{1}{2}(\varpi_1+\varpi_2)$ and $\omega_2=\varpi_1$.
%%end footnote
}\!\!.

Consider any parameter value $s_0$, and take a point $z_0\in\mathbb{C}$ for which $\wp(z_0)=f(s_0)$. Since both $f$ and $\wp$ satisfy the equation (\ref{eq:odekappabis}), we have $\wp'(z_0)=\pm f'(s_0)$. After possibly having reverted the equi-affine arc-length parameter, it can be assumed that the plus sign appears in this equation. Then, by integration of 
(\ref{eq:odekappabis}) for both $\wp$ and $f$, we find that $\wp(z_0+s)=f(s_0+s)$ for all real numbers $s$ in a neighbourhood of the origin. We conclude that
\[
f(s)=\wp(s-c_0)
\]
holds for all real $s$ in a certain interval and for a certain fixed number $c_0$, which can be assumed to be a pure imaginary.
Since $\wp(s-c_0)$ is a real number whenever $s$ is real, there are merely two different situations to consider: either $c_0=\varpi_2$ or $c_0=0$.

If $c_0=\varpi_2$ and $\Delta>0$, the curvature function will oscillate between two values (say, $q=\textrm{min}\,\kappa$ and $Q=\textrm{max}\,\kappa$) and the image in the phase-plane describes the closed branch of the cubic $\mathcal{C}$. Whenever the equi-affine curvature has the value $q$ or $Q$, the derivative $\kappa'$ will vanish, and since this derivative is given by (\ref{eq:odekappabis}), the numbers $g_2$, $g_3$ can be expressed in terms of $q$ and $Q$:
\begin{equation}
\label{eq:g2g3qQ}
g_2 = \frac{1}{9}\big(q^2+Q^2+q\,Q\big)
\qquad
\textrm{and}
\qquad
g_3 = \frac{1}{54}\big(q^2\,Q+q\,Q^2\big)\,.
\end{equation}

%%%begin figure : cubic mathcal{C}
\newcommand{\figwid}{0.193\textwidth}
\newcommand{\unitwid}{0.0125\textwidth}
\begin{figure}
\begin{center}
\includegraphics[bb=50 0 500 150,width=0.8\textwidth]{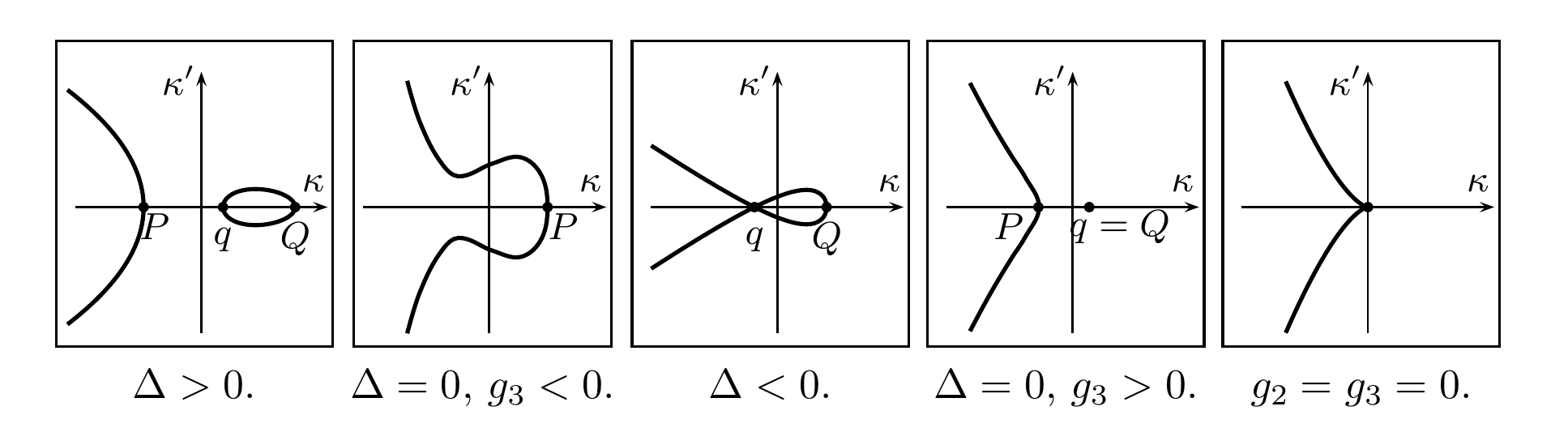}
\caption{The different possibilities for the cubic $\mathcal{C}$ in the phase plane, depending on the sign of $g_2$, $g_3$ and the cubic discriminant $\Delta=(g_2)^3-27\,(g_3)^2$.}
\label{fig:cubicC}
\end{center}
\end{figure}
%%%end figure : cubic mathcal{C}

If $c_0=\varpi_2$ and $\Delta<0$, then $\wp(s-c_0)=\wp(s-\varpi_1)$. After translation of the parameter $s$ by the real number $\varpi_1)$, this is already included in the next case.

If $c_0=0$, the curvature $\kappa(s)=-6\,\wp(s)$ will raise from $-\infty$ at $s=0$, have a maximum $P$ at $\varpi_1$, and fall again towards $-\infty$ for $s$ approaching $2\,\varpi_1$. 

Now that we have found the curvature function $\kappa(s)=-6\,\wp(s-c_0)$ of the curve, we have to find the co-ordinate functions $x$ and $y$, as a function of $s$, which satisfy (\ref{eq:defkappa}). The translates of these differentiated functions, which map $s$ to $x'(s+c_0)$ resp. $y'(s+s_0)$, satisfy the homogeneous second-order ordinary differential equation 
\[
F'' = 6\,\wp\,F\,,
\]
\textit{i.e.}, the equation of Lam\'e. Two independent complex solutions\footnote{Here $\varphi_2$ can be replaced by the second solution of Lam\'e's equation which is given in \cite{whittaker_watson} p.\ 459, ex.\ 29, in case this function is not a multiple of $\varphi_1$.} of this differential equation are
\begin{equation}
\label{eq:varphi}
\left\{
\begin{array}{rcl}
\varphi_1(z) & = & \displaystyle\frac{\partial\ }{\partial z} 
\left\lgroup \frac{\sigma(z+c)}{\sigma(z)}\,\mathrm{exp}\left(\left(\frac{-\wp'(c)}{2\,\wp(c)}-\zeta(c) \right)z\right) \right\rgroup\,; \\
\varphi_2(z) & = & \displaystyle\int_{z_0}^{z} \frac{1}{\big(\varphi_1(v)\big)^2 } \,\dd v\,\varphi_1(z)\,.\rule{0pt}{22pt} \\
\end{array}
\right.
\end{equation}
Here $c$ stands for a complex 
%%begin footnote
number\footnote{There are two possible ways to choose the constant $c$ in each period parallelogram in such a way that the condition $\wp(c)=\frac{-g_3}{g_2}$ is satisfied, but this choice does not influence the geometry of the curve.}
%%end footnote
for which $\wp(c)=\frac{-g_3}{g_2}$, and which may be taken in one of the forms $\varpi_1+d\,i$, $\varpi_2+d$, $d$ or $d\,i$ (for a real number $d$) if $\Delta>0$, and in one of the forms $d$ or $d\,i$ (for a real number $d$) if $\Delta<0$.
Further, $\sigma$ and $\zeta$ have their usual meaning as functions constructed from $\wp$ (see \cite{whittaker_watson},~Ch.~\textsc{xx}). The curve $(x,y)$ is now described by
\begin{equation}
\label{eq:xy}
\left\lgroup
\begin{array}{c}
x(s)\\
y(s)\rule{0pt}{15pt}
\end{array}
\right\rgroup
=
B
\left\lgroup
\begin{array}{c}
\displaystyle \int_{s_0}^{s} \varphi_1(t-c_0) \,\dd t \\
\displaystyle \int_{s_0}^{s} \varphi_2(t-c_0) \,\dd t \rule{0pt}{23pt}
\end{array}
\right\rgroup
+
\left\lgroup
\begin{array}{c}
x_0\\
y_0\rule{0pt}{15pt}
\end{array}
\right\rgroup
\end{equation}
for some numbers $x_0,y_0,s_0\in\mathbb{R}$ and a non-degenerate complex matrix $B$. Because of the fact $\varphi_1\,\varphi_2'-\varphi_2\,\varphi_1'=1$, this curve will be parametrised by equi-affine arc-length if and only if $\det\,B=1$.

The above construction can be simplified if the functions $\mathfrak{Re}\,\varphi_1(s-c_0)$ and $\mathfrak{Im}\,\varphi_1(s-c_0)$ are independent, for then we can write simply
\begin{equation}
\label{eq:xymoresimply}
\left\lgroup
\begin{array}{c}
x(s)\\
y(s)\rule{0pt}{15pt}
\end{array}
\right\rgroup
=
B
\left\lgroup
\begin{array}{c}
\displaystyle  \mathfrak{Re} \left\lgroup \frac{\sigma(s-c_0+c)}{\sigma(s-c_0)}\,\mathrm{exp}\left(\left(\frac{-\wp'(c)}{2\,\wp(c)}-\zeta(c) \right)(s-c_0)\right) \right\rgroup \\
\displaystyle  \mathfrak{Im} \left\lgroup \frac{\sigma(s-c_0+c)}{\sigma(s-c_0)}\,\mathrm{exp}\left(\left(\frac{-\wp'(c)}{2\,\wp(c)}-\zeta(c) \right)(s-c_0)\right) \right\rgroup \rule{0pt}{26pt}
\end{array}
\right\rgroup
+
\left\lgroup
\begin{array}{c}
x_0\\
y_0\rule{0pt}{15pt}
\end{array}
\right\rgroup
\end{equation}
where the matrix $B$ has now real entries, and should have a determinant appropriately chosen so as to obtain the equi-affine normalisation for the curve $(x,y)$.

\noindent\framebox{\textsc{Case} (A). $g_2\neq 0$, $\Delta> 0$, and the closed component of $\mathcal{C}$ is described.}\raisebox{-9pt}{\rule{0pt}{24pt}}\ \ 
Since the image of the curve $\gamma$ travels along the closed branch of $\mathcal{C}$ in the phase plane in 
Fig.~\ref{fig:cubicC} (left), there necessarily holds $c_0=\varpi_2$. We consider three subcases, depending on the sign of $q=\min\, \kappa$.

\noindent
\textbf{(A.1).} $q>0$. ---In this case it is no restriction to write $q=1$, which can always be achieved by a rescaling. Due to (\ref{eq:g2g3qQ}), the constant $Q>1$ determines the invariants $g_2,g_3$, which will be strictly positive in our case, and consequently the numbers $\varpi_1\in\mathbb{R}$ and $\varpi_2\in i\,\mathbb{R}$ can be found as well. There holds $c=\varpi_1+d\,i$, and formula (\ref{eq:xymoresimply}) is valid as well. 

Let us now search for a periodicity condition on $\gamma$. The period of $\gamma$ has to be a multiple of $2\,\varpi_1$ which is the period of $\kappa$. From the quasi-periodicity of $\sigma$ (see \cite{whittaker_watson},~\S\,22.421) can be deduced that $X(s)=x(s)+i\,y(s)$ satisfies
\begin{equation}
\label{eq:4momega1}
X(s+4\,m\,\varpi_1) = 
\textrm{exp}\left( -4\,m
\left\lgroup\left(\frac{\wp'(c)}{2\,\wp(c)} + \zeta(c) \right)\varpi_1
-\zeta(\varpi_1)\,c \right\rgroup
\right)\,X(s)\,.
\end{equation}
Thus the curve closes up if and only if 
\begin{equation}
\label{eq:closednesscondition}
\left\lgroup
\left(\frac{\wp'(c)}{2\,\wp(c)} + \zeta(c) \right)\varpi_1
-\zeta(\varpi_1)\,c
\right\rgroup
\frac{2\,i}{\pi}
=\frac{n}{m}
\end{equation}
for some natural numbers $n,m$. It should be noticed that imaginary part of the left-hand side of  (\ref{eq:closednesscondition}) vanishes automatically. A period for the curve is given by $4\,m\,\varpi_1$ but for $m$ even, $2\,m\,\varpi_1$ will also be a period of the curve.

%%
%M24_N17_thickness2_v9.pdf
%M29_N37_thick_2_v6.pdf
%N4M3_v6.pdf
%N5M4_v12.pdf
%+table
\begin{figure}
\begin{center}
\includegraphics[angle=0,width=0.48\textwidth]{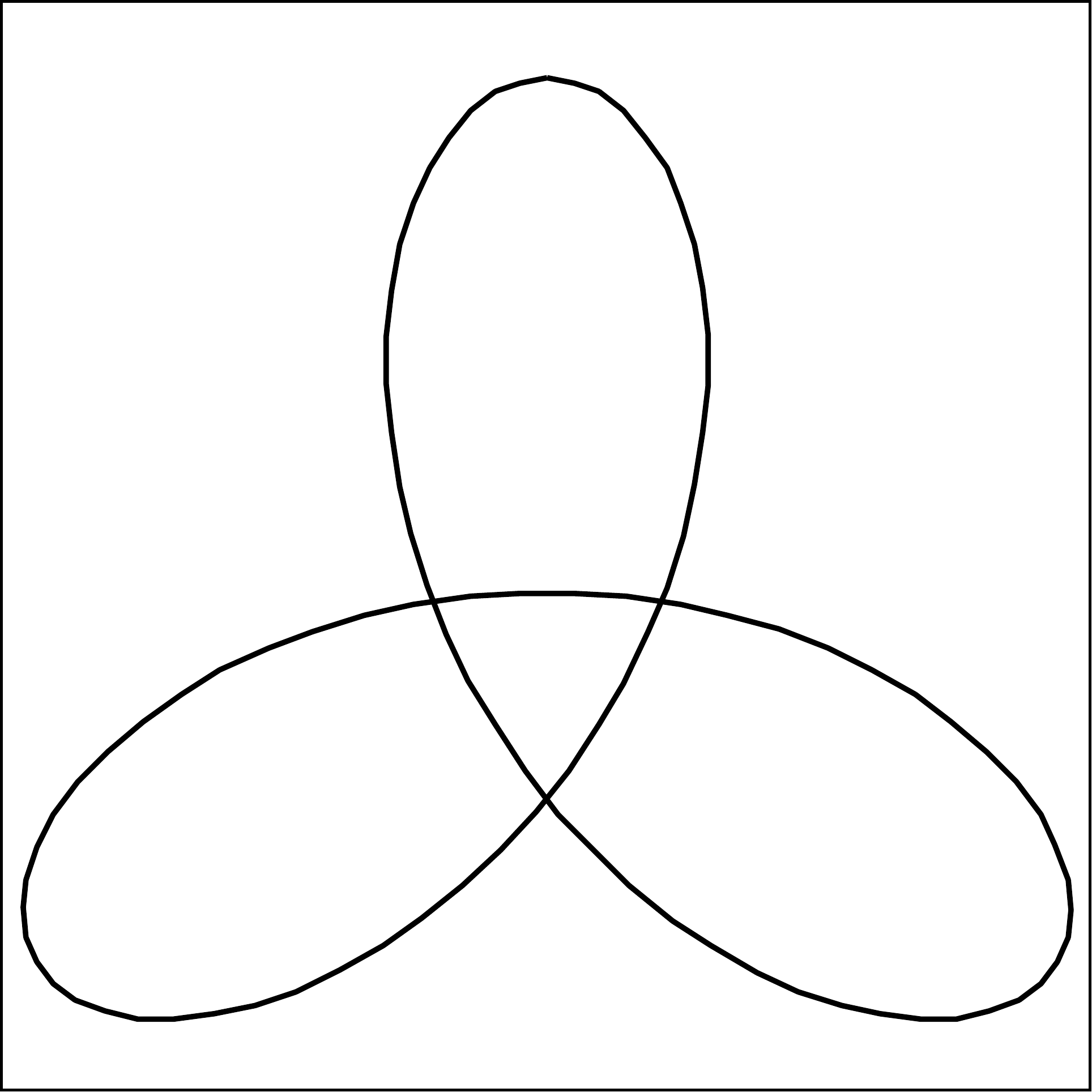}
\includegraphics[angle=0,width=0.48\textwidth]{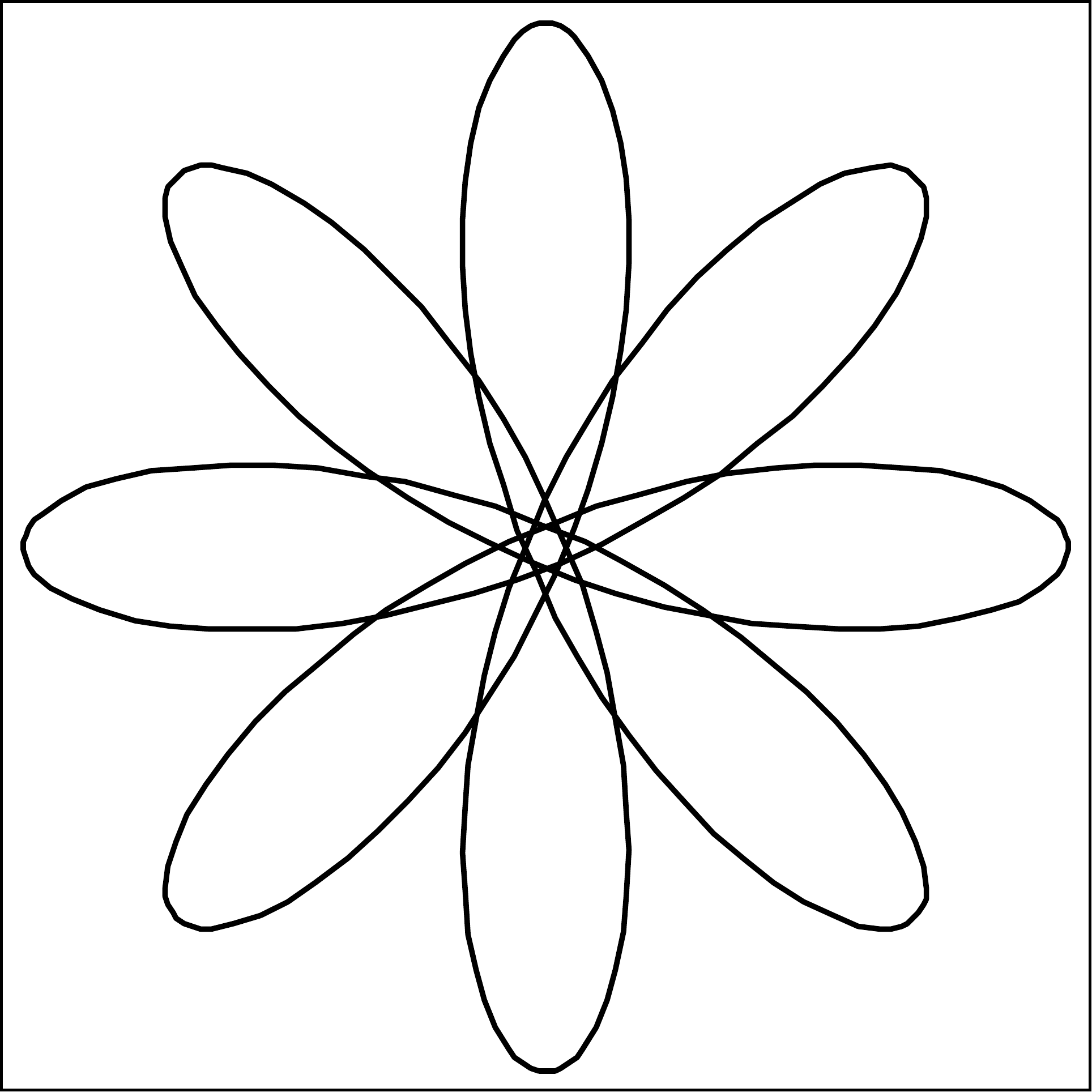}
\\
\vspace{1mm}
\includegraphics[angle=0,width=0.48\textwidth]{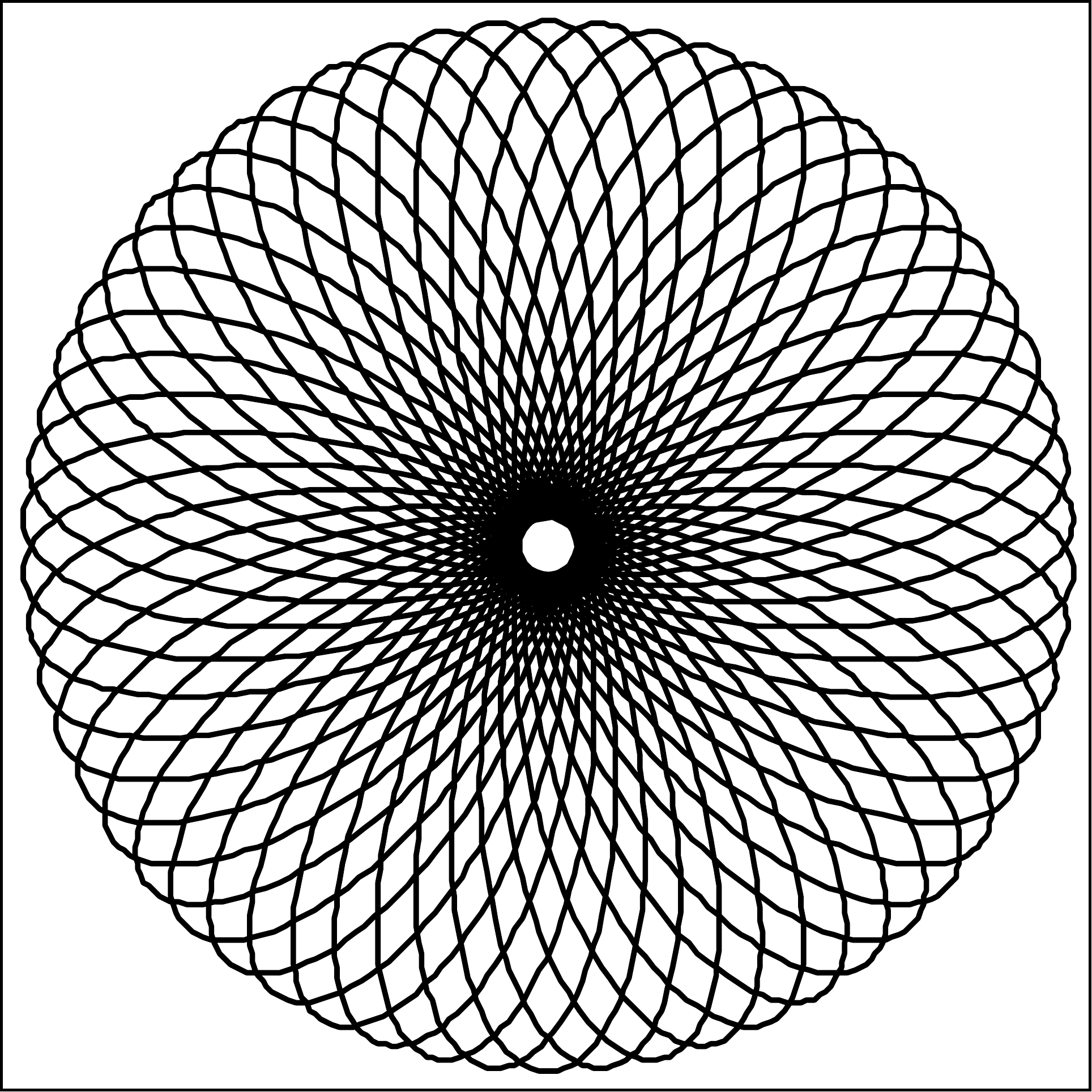}
\includegraphics[angle=0,width=0.48\textwidth]{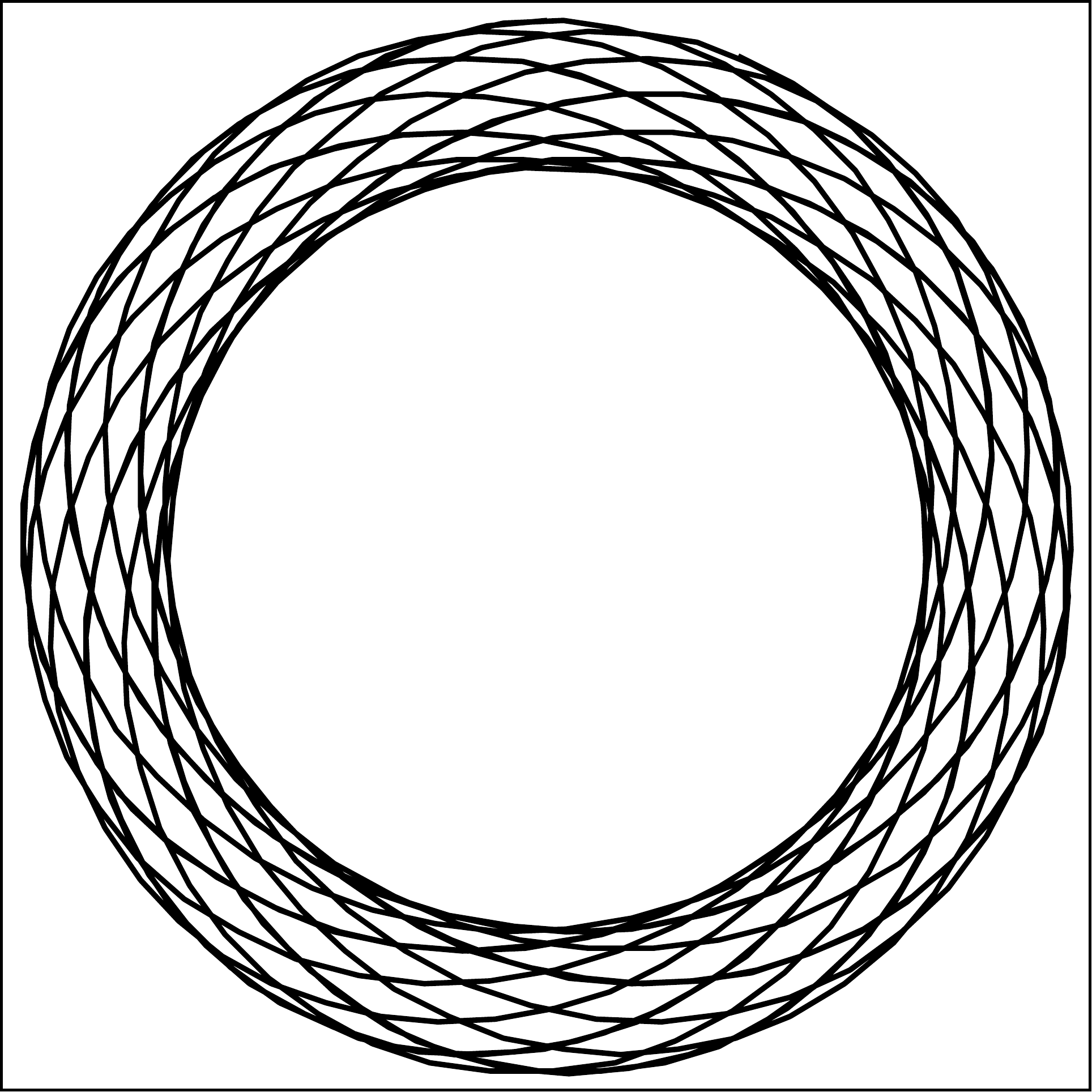}
\\
\vspace{1mm}
\begin{tabular}{|@{\ }l@{\ }|@{\ }l@{\ }|@{\ }l@{\ }|@{\ }l@{\ }|@{\ }l@{\ }|@{\ }l@{\ }|@{\ }l@{\ }|}
\hline
\hfill{}$m$\hfill{} & 
\hfill{}$n$\hfill{} & 
\hfill{}$Q$\hfill{} & 
\hfill{}$\varpi_1$\hfill{} & 
\hfill{}$\varpi_2$\hfill{} & 
\hfill{}$c$\hfill{} &
\hfill{}Figure\hfill{} 
\rule{0pt}{13pt} \\
\hline
\phantom{0}3 & \phantom{0}4 & \phantom{0}3.940\,854\,279 & 1.424\,009\,578
& $(1.670\,043\,233)\,i$ & $\varpi_1-(1.540\,700\,057)\,i$ & Top, left.\rule{0pt}{13pt} \\
%%%%%
\phantom{0}4 & \phantom{0}5 & \phantom{0}8.947\,959\,902  & 
1.009\,840\,213 & $(1.086\,362\,374)\,i$ & 
$\varpi_1-(1.058\,686\,673)\,i $ & Top, right. \\
%%%
%\phantom{0}5 & \phantom{0}6 & 15.258\,665\,309 &
%0.792\,229\,630 & $(0.827\,535\,480)\,i$ &
%$\varpi_1-(0.817\,641\,870)\,i$ & (\textit{Not depicted}.)\\
%%%%
29 & 37 & \phantom{0}6.926\,542\,623 &
1.129\,312\,548 & $(1.239\,778\,028)\,i$ &
$\varpi_1-(1.194\,744\,029)\,i $
& Bottom, left. \\
%%%%
17 & 24 & \phantom{0}1.244\,192\,459 &
2.097\,620\,948 & $(3.602\,731\,724)\,i$ &
$\varpi_1-(2.351\,154\,225)\,i$ & Bottom, right.\\
%%%%
\hline
\end{tabular}
\caption{Some closed curves from Case (A.1) and the corresponding data.}
\label{fig:closed1}
\end{center}
\end{figure}
\begin{figure}
\begin{center}
\includegraphics[angle=0,width=0.48\textwidth]{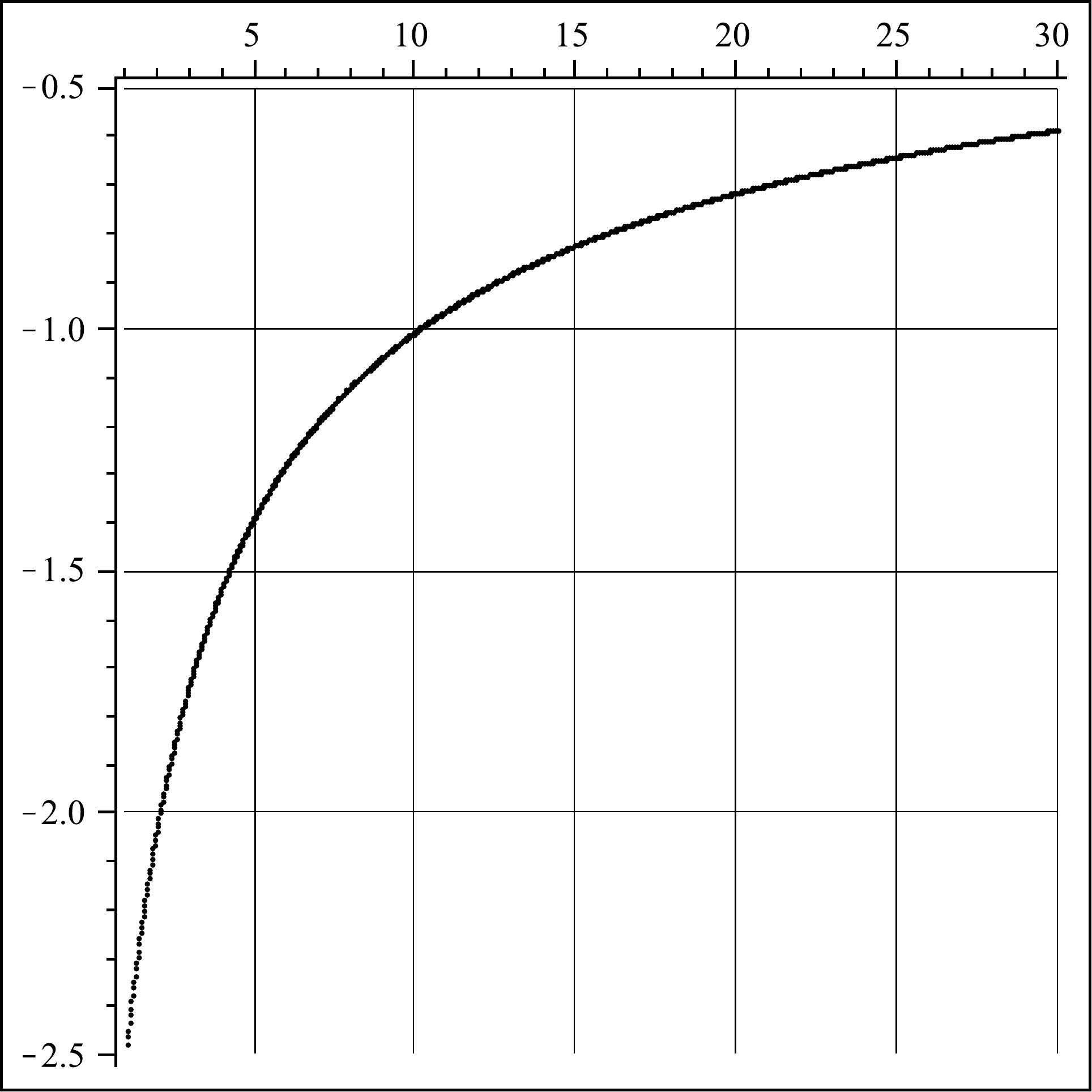}
\includegraphics[angle=0,width=0.48\textwidth]{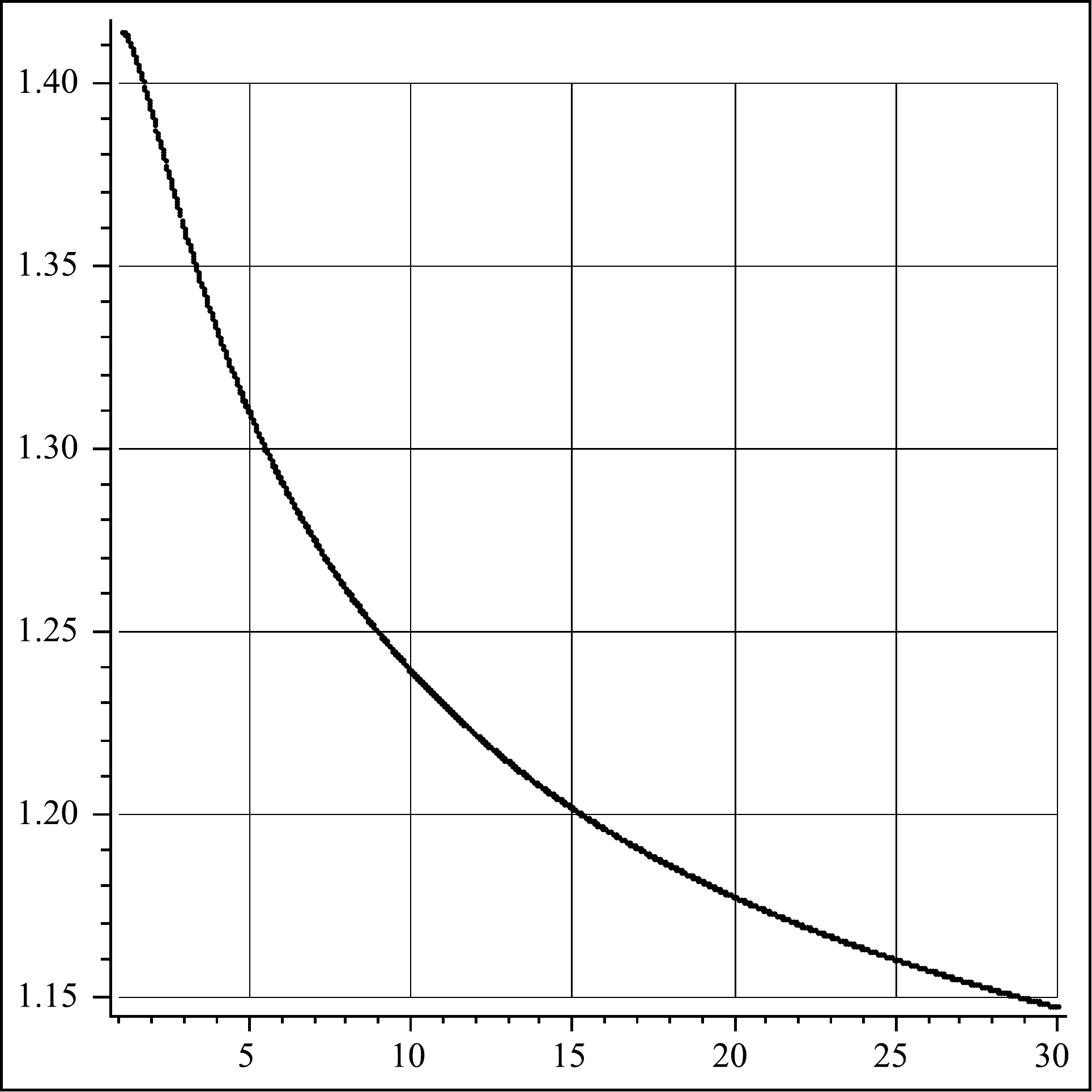}
\caption{The closedness condition for Case (A.1).
The horizontal axis represents the value $Q=\max\kappa$ for a curve with $q=1$. For given $q$ and $Q$, the values $g_2$, $g_3$, $\varpi_1$, and $\varpi_2$ are known. The value $d$ which can be read off from the left figure determines a complex number $c=\varpi_1+d\,i$ which satisfies $\wp(c)=\frac{-g_3}{g_2}$. In the right picture, the left-hand side of equation (\ref{eq:closednesscondition}) is plotted as a function of $Q$. If this value is a rational number, the curve will be closed.
}
\label{fig:QMN}
\end{center}
\end{figure}
\begin{figure}
\begin{center}
\framebox{
      \includegraphics[angle=90,height=0.37\textwidth,width=0.98\textwidth]{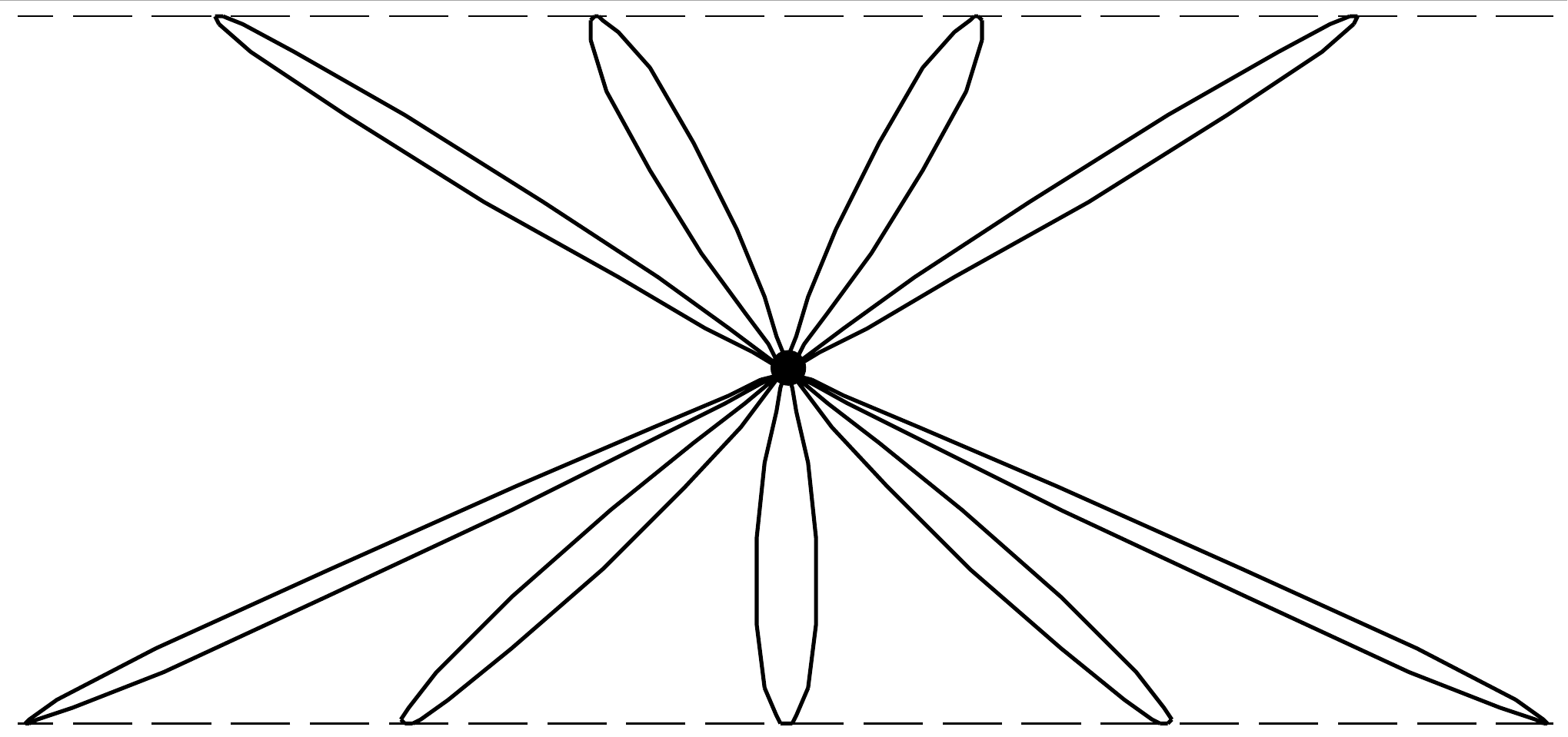}
}
\caption{A part of the curve, arising from $q=0$ and $Q>0$, as in Case (A.2). Parts of the two dashed lines are described by $\beta$, and the dot in the middle represents the origin.
}
\label{fig:qnul}
\end{center}
\end{figure}
%%

%%BEGIN FIGURE, q=-1, Q different values.
%%-----------------
%%q=-1,Q=20:
%%qmin_7nov_v3.pdf 
%%qmin_8nov_v18.pdf
%%-----------------
%%q=-1,Q=6:
%%qmin_8nov_Q6_v31.pdf
%%qmin_8nov_Q6_v23.pdf
%%-----------------
%%q=-1,Q=2.25:
%%qmin_9nov_Q2point25_v19.pdf
%%qmin_9nov_Q2point25_v02.pdf
%%-----------------
\begin{figure}
\begin{center}
\includegraphics[height=0.31\textwidth,width=0.31\textwidth]{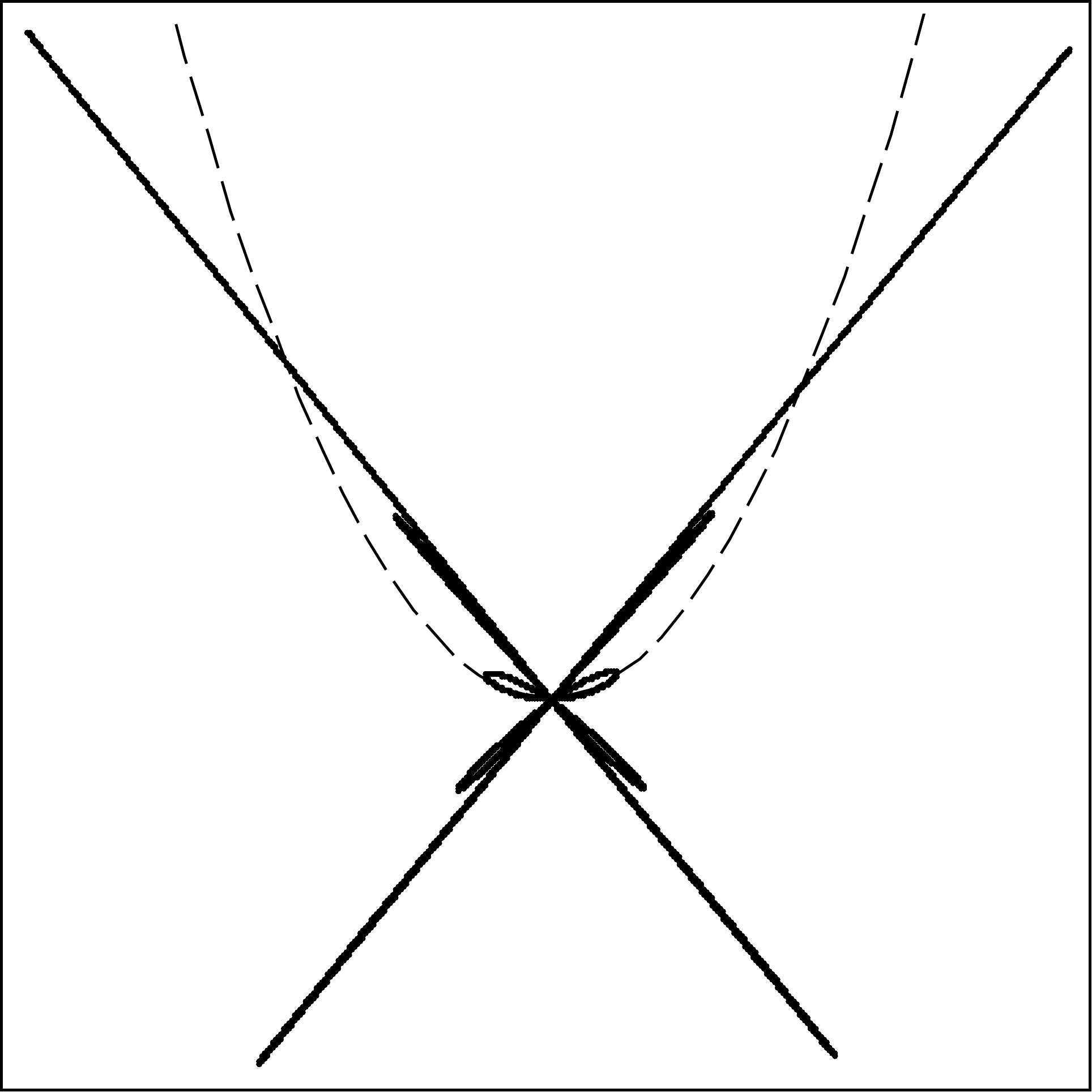}
\includegraphics[height=0.31\textwidth,width=0.31\textwidth]{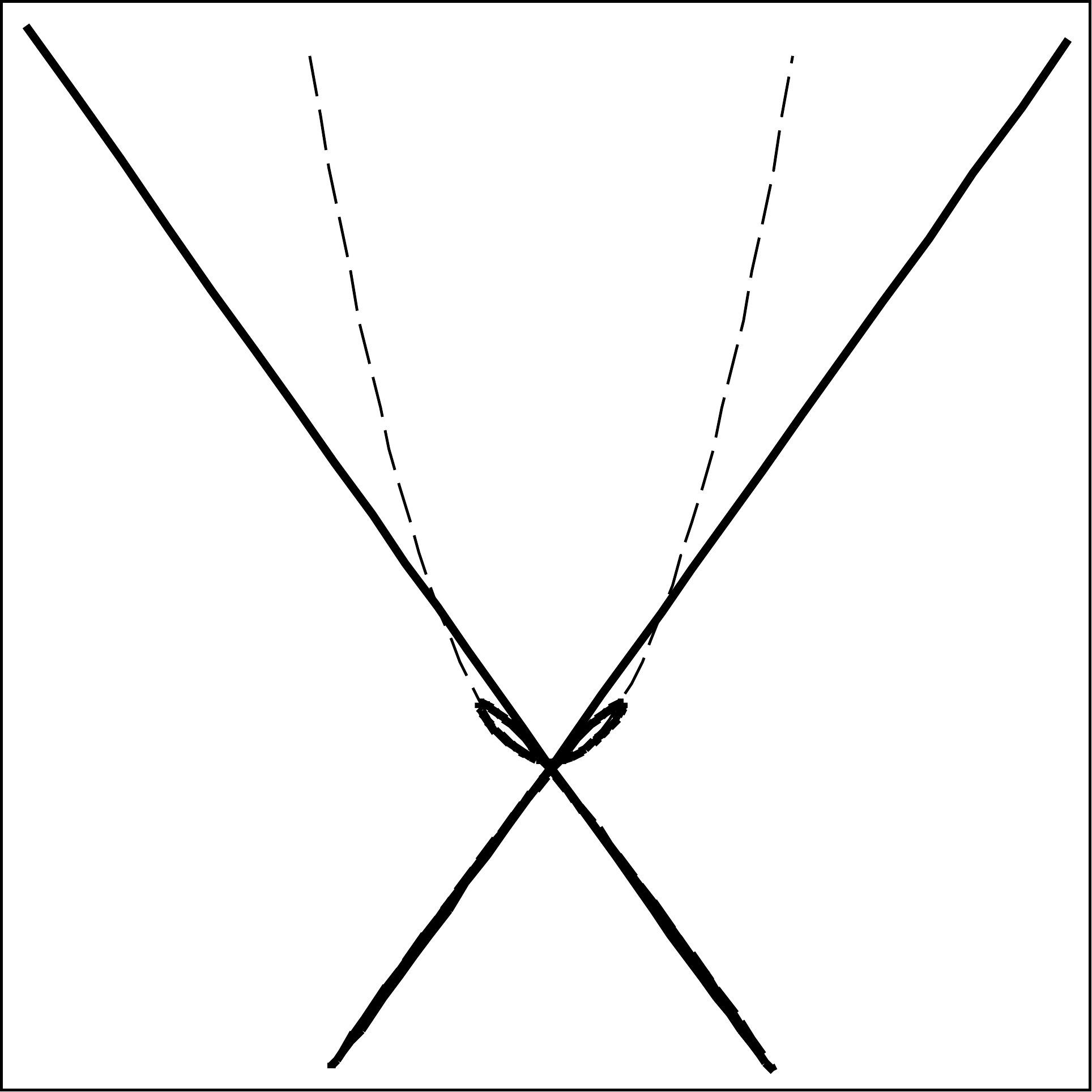}
\includegraphics[height=0.31\textwidth,width=0.31\textwidth]{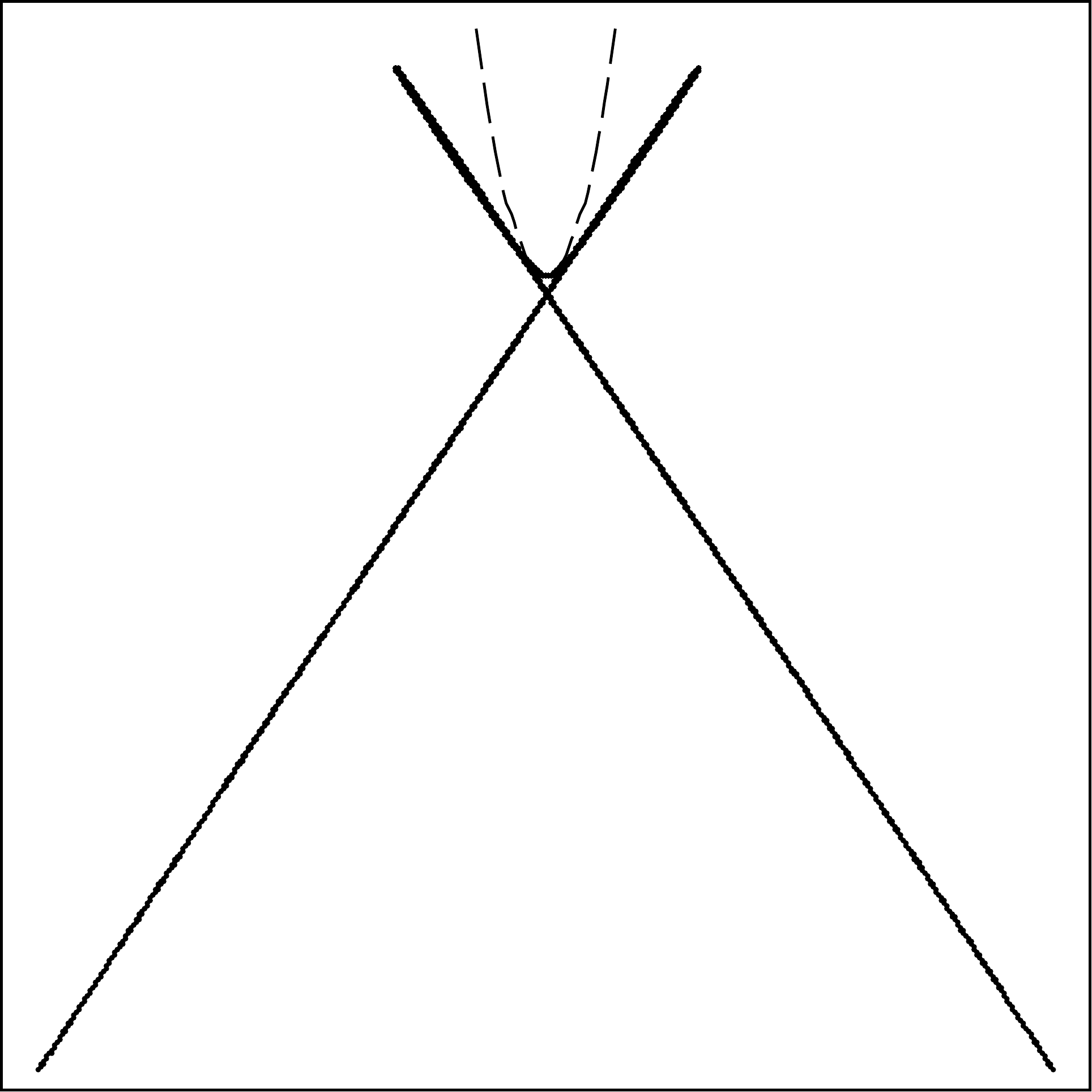}\\
\rule{0pt}{0.313\textwidth}
\includegraphics[height=0.31\textwidth,width=0.31\textwidth]{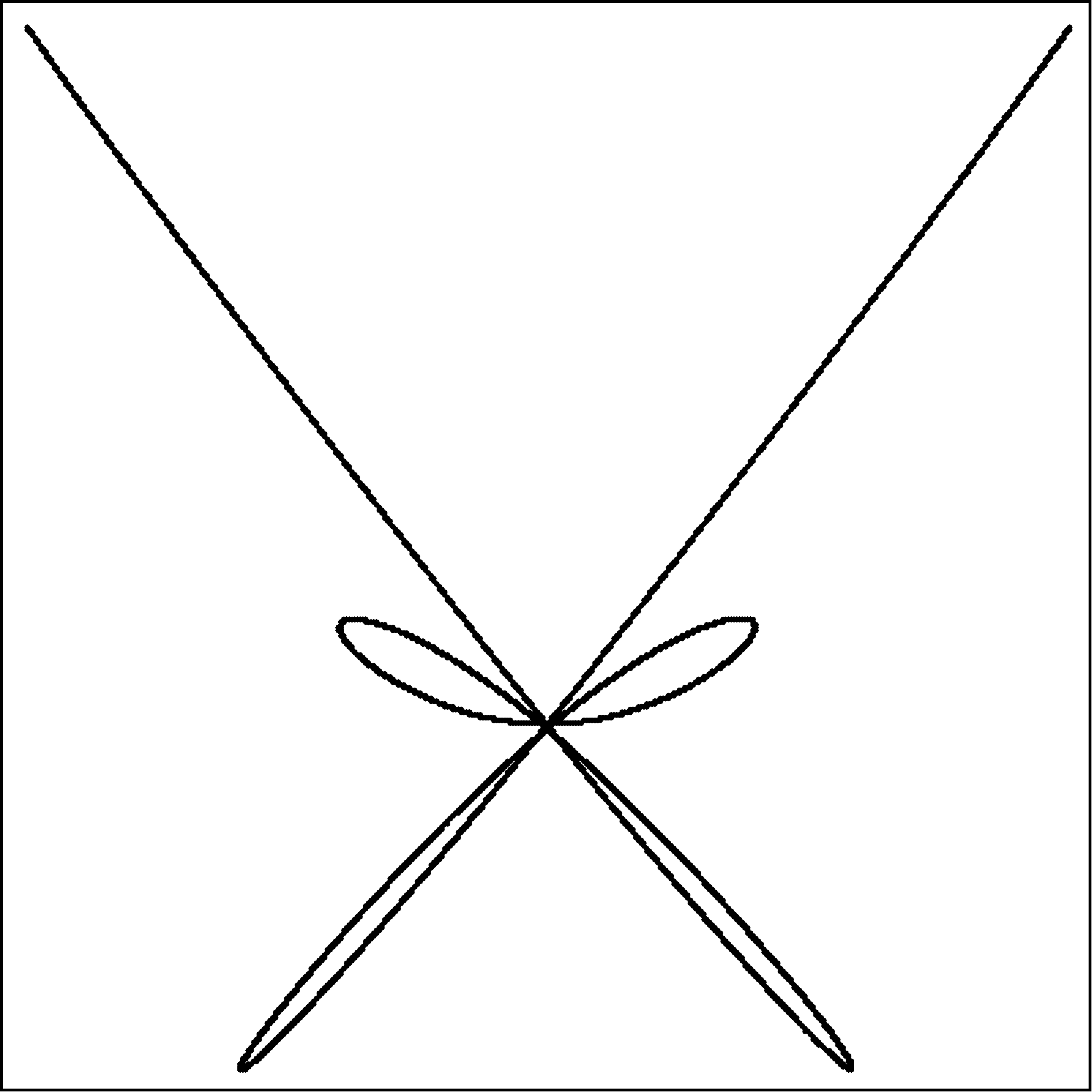}
\includegraphics[height=0.31\textwidth,width=0.31\textwidth]{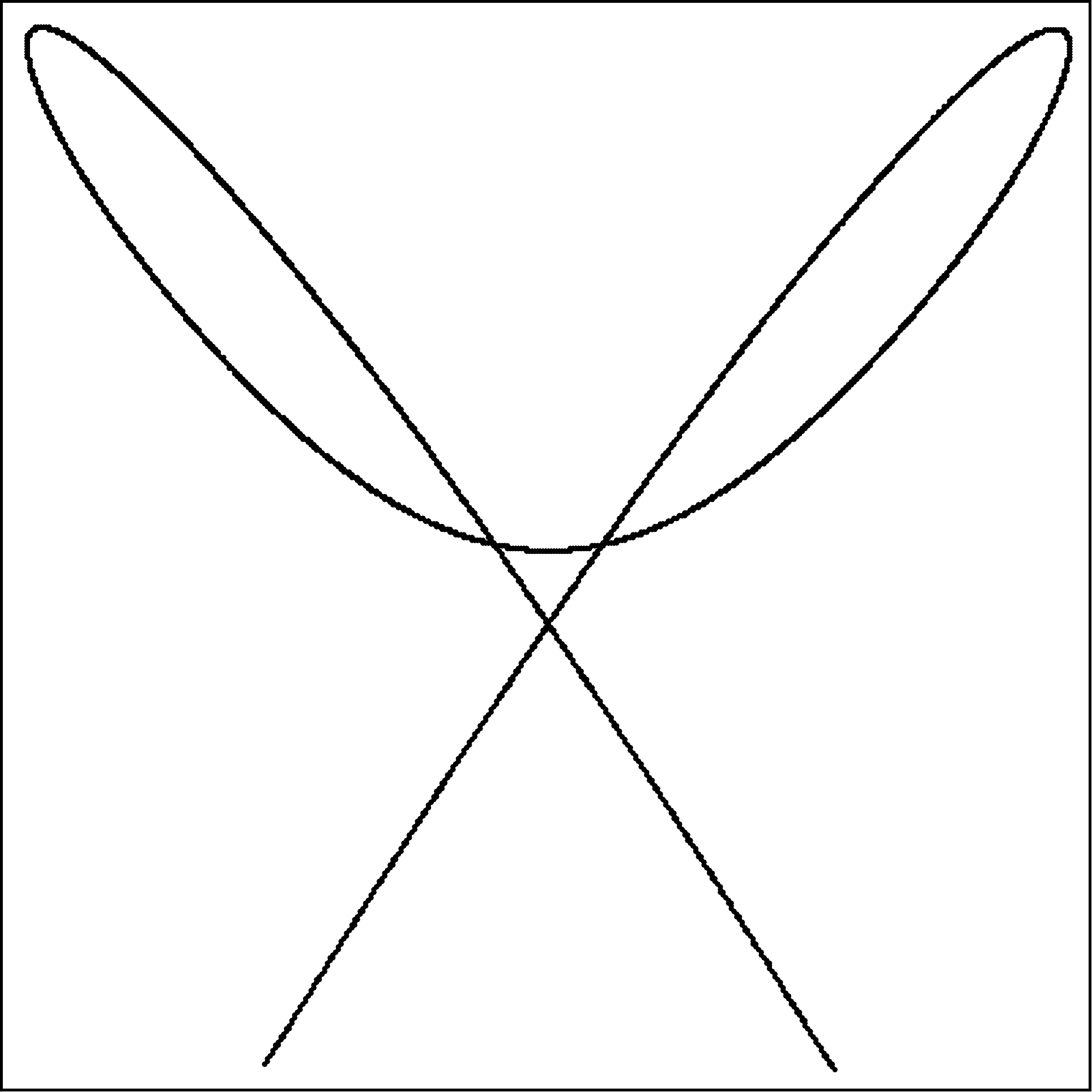}
\includegraphics[height=0.31\textwidth,width=0.31\textwidth]{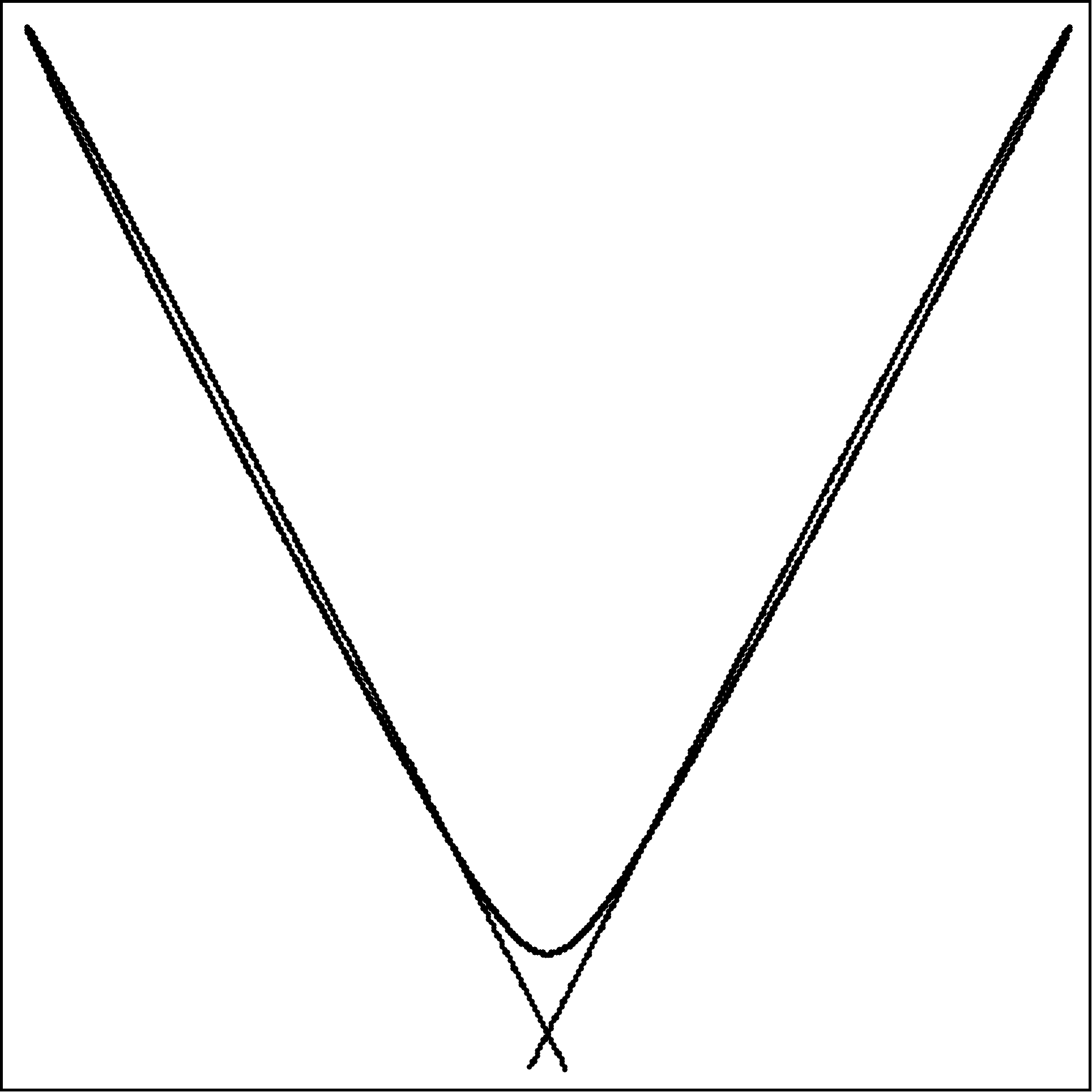}
\rule{0pt}{0.313\textwidth}
\caption{Case (A.3). Three instances of a curve obtained by choosing $q=-1$: the pictures on the second row are always enlarged parts of the picture on the first row, where the osculating parabola has been drawn as well at a point where $\kappa$ achieves a minimum.
Left: $Q=20$; top figure: 9 periods of $\kappa$ are displayed; bottom figure: 5 periods of curvature. Middle: $Q=6$; top resp.\ bottom figure: $4.7$ resp.\ 2.2 periods of curvature. Right: $Q=2.25$; top resp.\ bottom figure: $2.1$ resp.\ $1.7$ periods of curvature.}
\label{fig:qmin}
\end{center}
\end{figure}
%%END FIGURE, q=-1, Q different values.

Some values for which the condition (\ref{eq:closednesscondition}) is satisfied, which have been numerically obtained, are represented in a table together with the corresponding curves in Figure~\ref{fig:closed1}. 
For a given $Q>1$ which determines, together with $q=1$, the invariants $g_2$, $g_3$ and half-periods $\varpi_1$, $\varpi_2$, can be found a number $c=\varpi_1+d\,i$ for which  $\wp(c)=\frac{-g_3}{g_2}$ holds, and then it can be verified whether or not the closedness condition
(\ref{eq:closednesscondition}) is satisfied.
Now the data in the table accompanying Figure~\ref{fig:closed1}
have been found by, for a prescribed pair of integers $m$ and $n$, iteratively approaching the number $Q>1$ for which the corresponding data satisfy the closedness condition (\ref{eq:closednesscondition}) more and more precisely.
The collected numerical data \textit{suggest} that for prescribed $m$, $n$, a corresponding value of $Q$ can be found for which the curve will be closed, for this Case (A.1), if and only if $\frac{n}{m}\in\left]\,1\,,\,\sqrt{2}\,\right[$. The right part of Figure~\ref{fig:QMN} can be seen as a plot of the values of $\frac{n}{m}$ versus  $Q$.

The critical observer will have noticed that the curves, plotted in Figure~\ref{fig:closed1}, resemble the familiar hypotrochoids suspicously well. It should be stressed that these are different curves, as follows from the fact that the hypotrochoids do not satisfy the Euler--Lagrange equation (\ref{eq:odekappabis}). 

Let us now explain why the curves in Figure~\ref{fig:closed1} have such a ``Euclidean'' symmetry. The curve $s\mapsto \gamma(s-2\,\varpi_1)$ has the same curvature function as $\gamma$, and consequently there exists an equi-affine congruence $\psi$ under which this curve is sent to $\gamma$. Since $\gamma$ closes up after $4\,m\,\varpi_1$, there necessarily holds $\psi^{2\,m}=\textbf{1}$, which implies that the linear part of this mapping has two complex-conjugate eigenvalues. The points of the form $\psi^{\ell}(p)$, for a fixed point $p$, are then easily seen to lie on an ellipse.

We conclude that the group of orientation-preserving equi-affine symmetries of a closed curve consists of the orientation-preserving equi-affine transformations which preserve $2\,m$ points separated at equal equi-affine distance on an ellipse. 
With respect to the Euclidean metric for which this ellipse becomes a circle, this symmetry group becomes a finite group of rotations over an angle which is a multiple of $\frac{\pi}{m}$. 

As such, intending to most appropriately visualise these affine curves on a sheet of paper which was endowed with a Euclidean metric anyway, I have opted to select this affine representative of the curve among all its affine transforms which solve the variational problem equally well, for which the ellipse joining the points of maximal equi-affine curvature becomes a circle. 
It is in this way that the geometrical properties of the curves under consideration are most enjoyed by an audience living in a Euclidean, and not an affine, world.

\noindent
\textbf{(A.2).} $q=0$. ---From (\ref{eq:g2g3qQ}) follows that $g_2=\frac{Q^2}{9}$ and $g_3=0$. The choice of $Q$ merely influences the curve by an affine mapping, and therefore we will assume $Q=1$. The curve can most conveniently be described by making use of the fact that (up to a rescaling of the parameter) $\beta=\frac{1}{\sqrt{\kappa}}\,\gamma$ describes a centro-affinely parametrised straight line. In this way we find, up to a full-affine motion,
\begin{equation}
\label{eq:a2}
\big(x(s),y(s)\big) = \pm\sqrt{\kappa(s)}\,\big(1,s\big),
\end{equation}
where $\kappa(s)=-6\,\wp(s-\varpi_2)$. It should be remarked that the curve is not smooth at the points $s=(2\,\ell+1)\,\varpi_1$ (for $\ell\in\mathbb{Z}$), unless a change sign in the right-hand side is induced at these points. As is clear from Figure~\ref{fig:qnul}, this enforces the path described by $\beta$
to be split in two parallel lines.

\noindent
\textbf{(A.3).} $q<0$. ---The left-most intersection point of $\mathcal{C}$ with the axis $\kappa'=0$, which has been denoted by $P$ in Fig.~\ref{fig:cubicC}, is given by $P=-q-Q$ and is strictly smaller then $q$. Therefore, assuming $q=-1$, we necessarly have $Q>2$, and from (\ref{eq:g2g3qQ}) follows that $g_2>0$ and $g_3<0$. The constant $c$ can be taken in the form $c=\varpi_2+d$ where 
$0\leqslant d \leqslant \varpi_1$. The $x$-co-ordinate of the curve, which is given by (\ref{eq:xy}), satisfies (\ref{eq:4momega1}). A simple calculation shows that
\[
\hspace{-3mm}\left\lgroup
\left(\frac{\wp'(c)}{2\,\wp(c)} + \zeta(c) \right)\varpi_1
-\zeta(\varpi_1)\,c
\right\rgroup
\frac{2\,i}{\pi}
= 1 + 
\left\lgroup
\varpi_1\sqrt{\frac{-g_3}{g_2}} - \zeta(\varpi_1)\,d+\varpi_1\,\zeta(\varpi_2+d)-
\varpi_1\,\zeta(\varpi_2)
\right\rgroup \frac{2\,i}{\pi}\,,
\]
and since the quantity between brackets in the right-hand side never vanishes, as can be shown numerically, the curve can never be periodic. Some of these curves are displayed in Fig.~\ref{fig:qmin}.

\noindent\framebox{\textsc{Case} (B). $g_2\neq 0$, $\Delta> 0$, and the non-closed branch of $\mathcal{C}$ is described.}\raisebox{-9pt}{\rule{0pt}{24pt}}\ \ After possibly having rescaled the curve, it can be assumed that the first ordinate of left point of intersection between the horizontal axis and the curve $\mathcal{C}$ in the phase plane is $P=-1$, \textit{i.e.}, the maximal curvature of the corresponding curve in $\mathbb{A}^2$ is $-1$. Keeping $q$ and $Q$ as notation for the two other points of intersection (which are not reached in this case), we can choose 
$q\in\left]\,-1\,,\,\frac{1}{2}\,\right[$, whereas $Q=1-q$.

\noindent
\textbf{(B.1).} $0<q<\frac{1}{2}$. ---The curve can be described by equation (\ref{eq:xymoresimply}).

\noindent
\textbf{(B.2).} $q=0$. ---An adaption of equation (\ref{eq:a2}) from Case (A.2) is valid.

\noindent
\textbf{(B.3).} $-1<q<0$. ---Such curves can be described by equation (\ref{eq:xy}). 

As can be seen on Figure~\ref{fig:B}, the curve follows its osculating hyperbola at its symmetrical point fairly close in all three cases.

%%Pmin1_v19.pdf (aff_k_case_VI_n.mw)
%%
\begin{figure}
\begin{center}
\includegraphics[angle=0,width=0.48\textwidth]{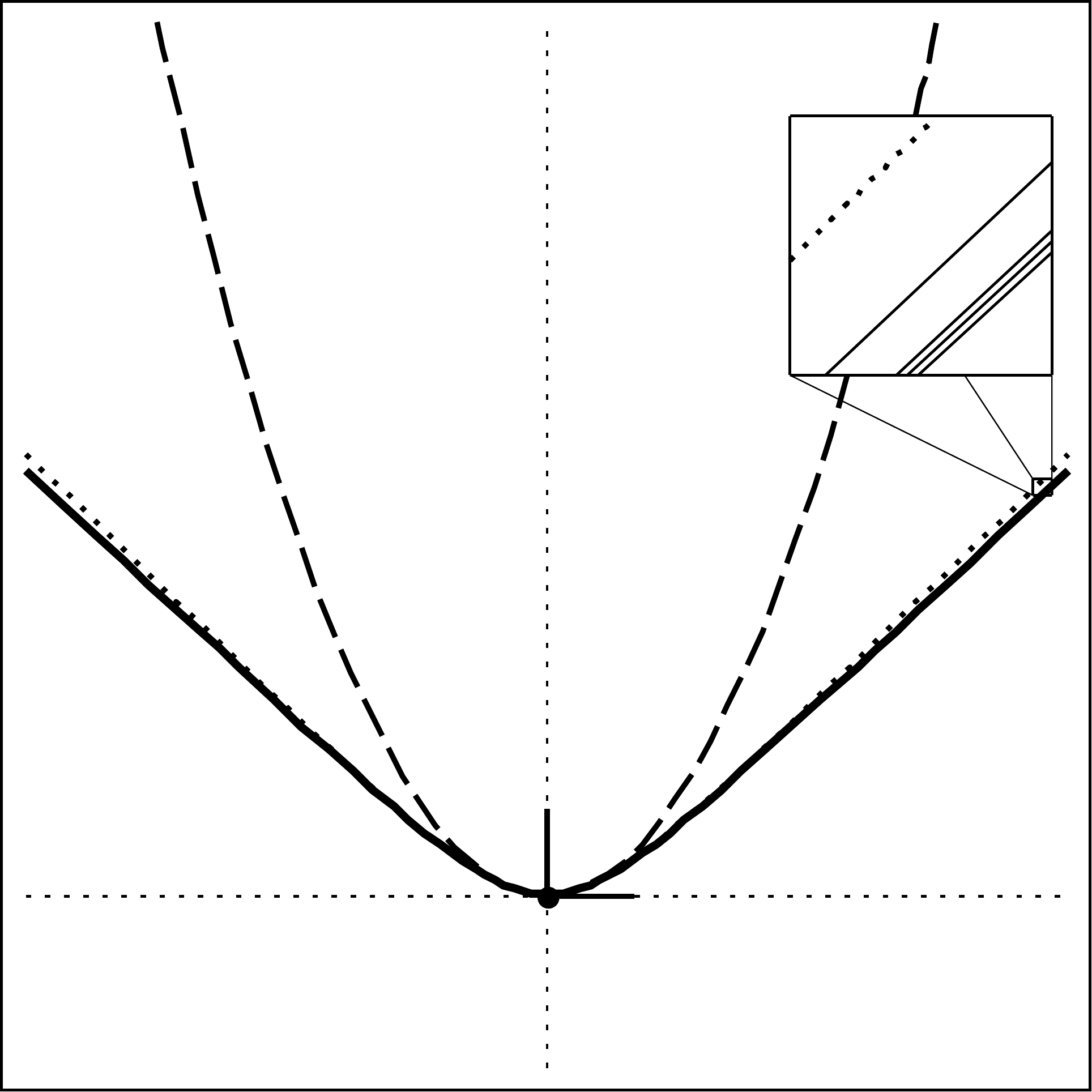}
\caption{Four of the curves described in Case (B) are displayed. Since these curves nearly coincide they are represented by a single thick line in the  main picture. They have been drawn in such a way that they share the equi-affine tangent and normal vector (displayed as short, black lines) at the point where their curvature takes the maximal value -1. Their common osculating parabola (dashed) and hyperbola (dotted) is displayed as well.
In the magnified parallelogram can be seen, from left-up to right-down, a portion of this hyperbola, and the four curves obtained by choosing $q=-0.5$, $0$, $0.15$ and $0.49$, respectively.
}
\label{fig:B}
\end{center}
\end{figure}

\noindent\framebox{\textsc{Case} (C). $g_2\neq 0$ and $\Delta< 0$.}\raisebox{-9pt}{\rule{0pt}{24pt}}\ \  Let us denote $P$ for the only real point of intersection between $\mathcal{C}$ and the horizontal axis in the phase plane, whereas the two remaining complex intersection points are $\frac{-P}{2}\pm\tau\,i$ for a number $\tau>0$. The invariants $g_2$ and $g_3$ are given by
\begin{equation}
\label{eq:g2g3Ptau}
g_2 = \frac{1}{36}\big(3\,P^2-4\,\tau^2\big)
\qquad
\textrm{and}
\qquad
g_3 = \frac{-P}{216} \big(P^2+4\,\tau^2\big)\,.
\end{equation}
After possibly having rescaled the curve $\gamma$, there holds $P=1$, $0$ or $-1$, whence the five cases below are seen to cover all possibilities.

\noindent
\textbf{(C.1).} $P=1$ and $\tau>\frac{\sqrt{3}}{2}$. ---Formula (\ref{eq:xymoresimply}) holds. See Figure~\ref{fig:C123}, left.

\noindent
\textbf{(C.2).} $P=1$ and $\tau<\frac{\sqrt{3}}{2}$. ---See Figure~\ref{fig:C123}, left.

\noindent
\textbf{(C.3).} $P=0$. ---Formula (\ref{eq:a2}) is valid. See Figure~\ref{fig:C123}, middle.

\noindent
\textbf{(C.4).} $P=-1$ and $\tau>\frac{\sqrt{3}}{2}$. ---See Figure~\ref{fig:C123}, right.

\noindent
\textbf{(C.5).} $P=-1$ and $\tau<\frac{\sqrt{3}}{2}$. ---See Figure~\ref{fig:C123}, right.

%%deltaNeg4c_combi_v1.pdf  -- combineert 3.1 en 3.2.
%%pnul_v1.pdf  -- 3.3
%%deltaNeg_8_v4.pdf  -- combineert 3.4 en 3.5.

\begin{figure}
\begin{center}
\includegraphics[angle=0,width=0.315\textwidth]{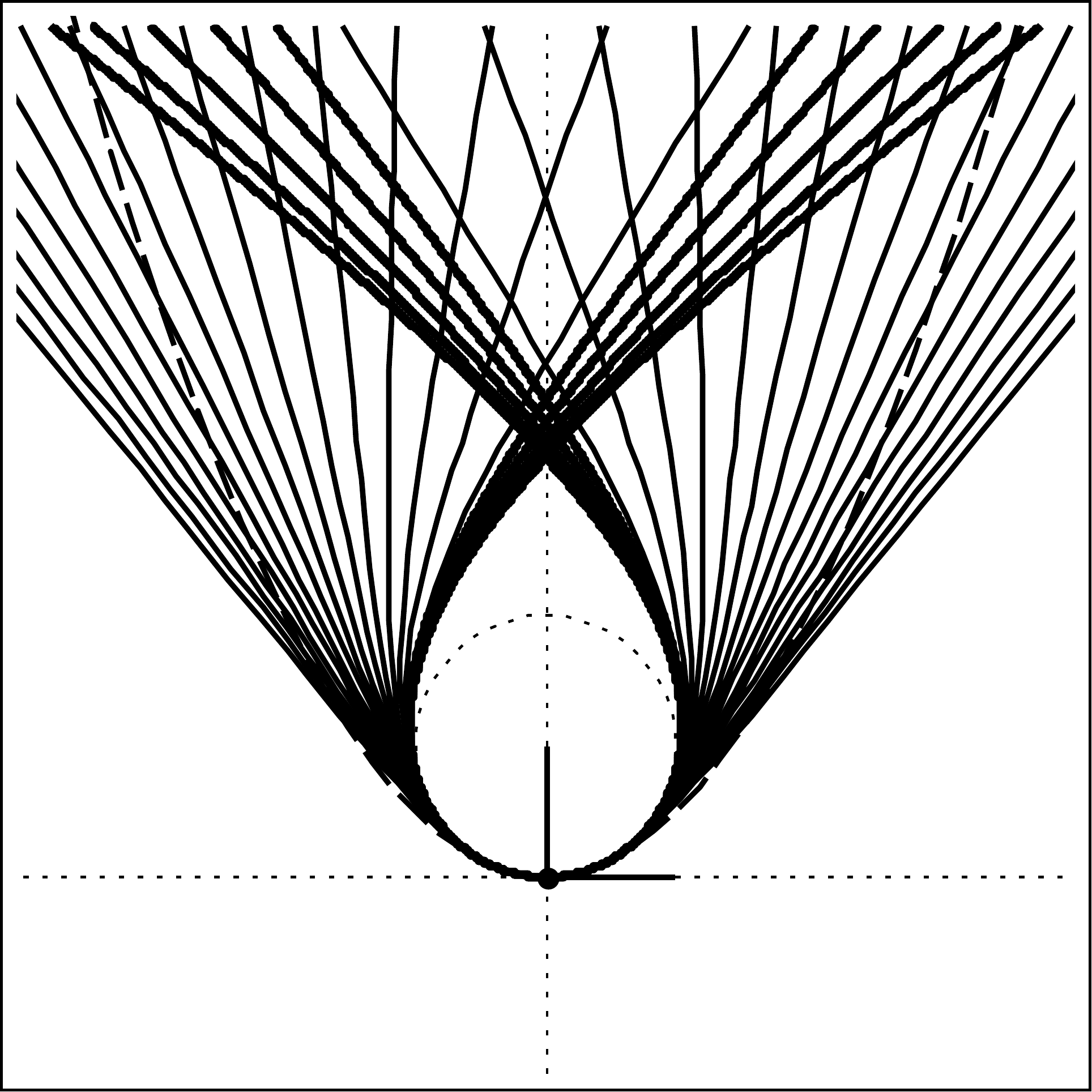}
\includegraphics[angle=0,width=0.315\textwidth]{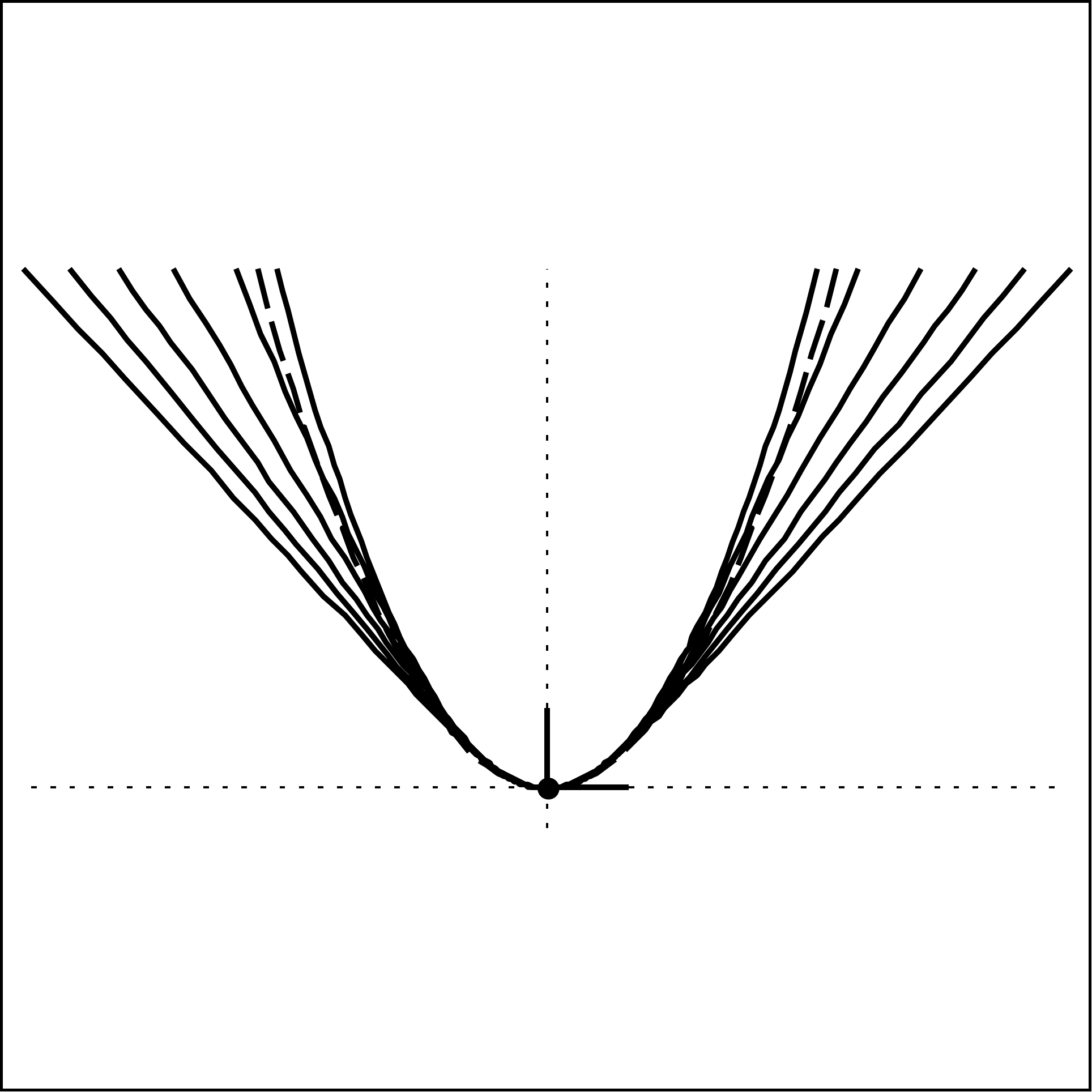}
\raisebox{-0.006\textwidth}{\includegraphics[angle=0,width=0.328\textwidth]{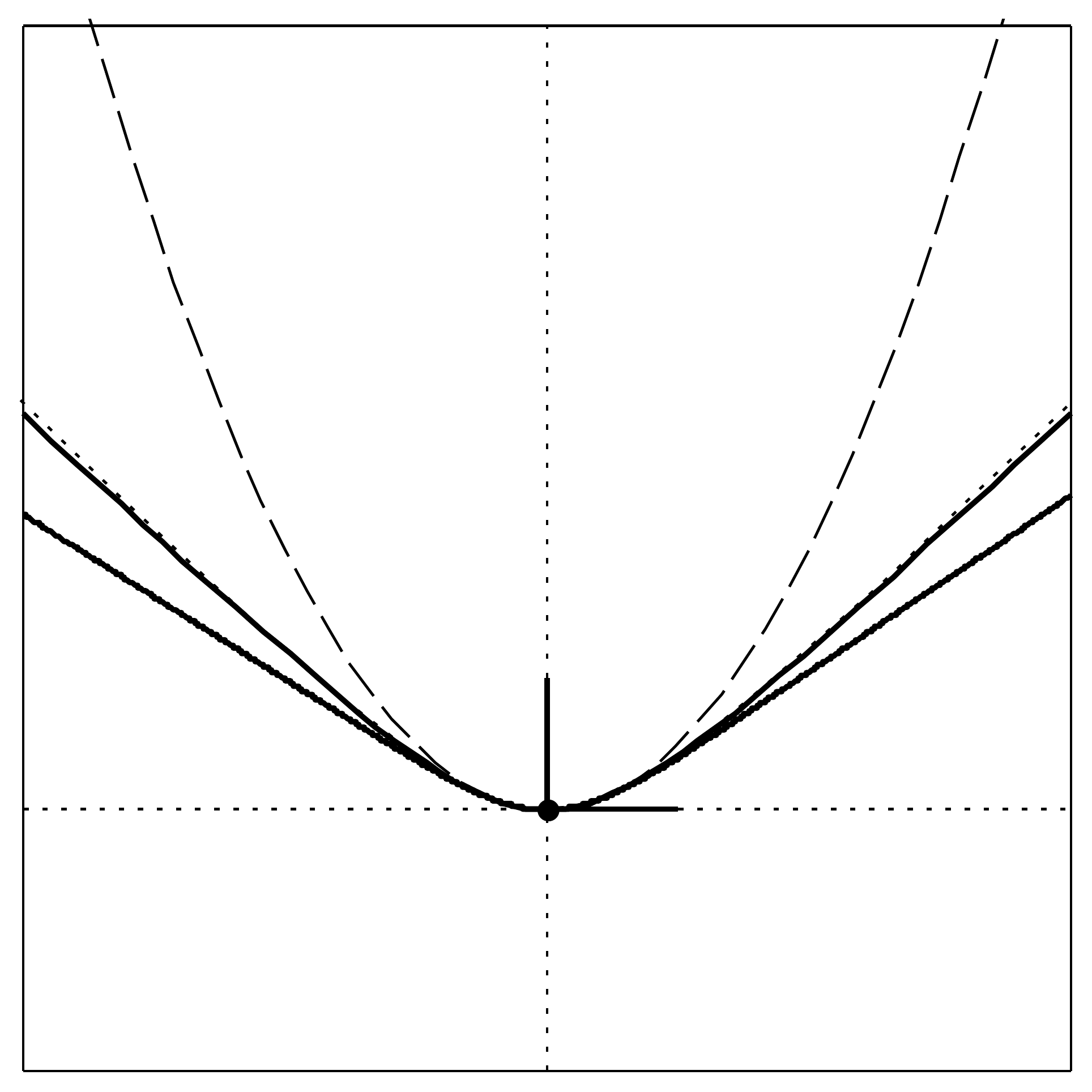}}
\caption{The left picture contains the curves of Case (C.1) with $\tau=\frac{7}{8}+\frac{4}{10}\,i$ for $i=0,1,2\ldots 15$ (in thin solid line) and the curves of Case (C.2) with $\tau=\frac{i}{7}$ for $i=1,2\ldots 5$ (in thick solid line). Middle: The curves of Case (C.3) where $\tau$ runs over the values $\frac{1}{10},1,2,3,4,5$. Remark that the parameter $\tau$ is merely a rescaling parameter in this case. The right picture displays two curves of Case (C.4) and (C.5), where $\tau=8$ and $\tau=\frac{1}{8}$, resp. In the latter case the curve nearly coincides with the osculating hyperbola at the point of maximal curvature.}
\label{fig:C123}
\end{center}
\end{figure}
%%

%%%%%%%%%%%%%%%%
\noindent\begin{tabular}
{@{\hspace{-0.002\textwidth}}p{0.74\textwidth}@{\hspace{0.05\textwidth}}p{0.2\textwidth}}
%%begin linkerdeel tabel
\noindent\framebox{\textsc{Case} (D). 
%%my notes: Case 2
$g_3 < 0$ 
and $\Delta = 0$.}\raisebox{-9pt}{\rule{0pt}{24pt}}\ \ 
Define $E=\sqrt[3]{g_3}<0$. The points of intersection of the cubic $\mathcal{C}$ with the horizontal axis in the phase plane have abscis $3\,E$ (as double point of $\mathcal{C}$) and $-6\,E$. The single curve $\mathcal{C}$ corresponds to four curves in the affine plane, as indicated in the picture right. The curves indicated by (a) and (d) are the same, up to a reversion of the arc-length parameter, whereas the single point (b)  in the phase plane represents a hyperbola. Therefore there remain only two cases to consider. In case (a), the curvature function can be found by integration of (\ref{eq:odekappabis}), with the following result: 
%%einde linkerdeel tabel
&
%%begin rechterdeel tabel = figuur
\rule{8cm}{0pt}\raisebox{-20mm}{\includegraphics[bb=250 50 400 150,width=0.3\textwidth]{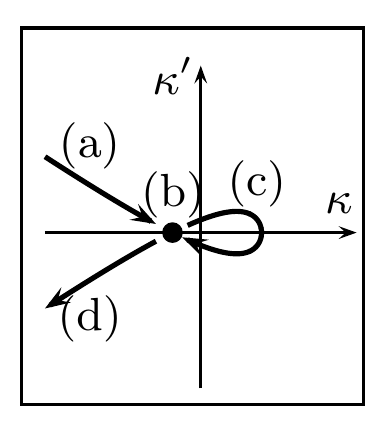}}
%%einde rechterdeel tabel = figuur
\end{tabular}
%%%%%%%%%%%%%%%%%%%%%%
\[
\kappa(s) = 
9\,E\,\left(\textrm{cotanh}\left(\sqrt{\frac{-3\,E}{2}}\,s\right)\right)^2-6\,E\,.
\]
Two solutions of the differential equation 
$\varphi''=-\kappa\,\varphi$ are given by
\[
\left\{
\begin{array}{rcl}
\varphi_1(s) &=& \exp\left(\sqrt{-3\,E}\,s\right)\left\lgroup
1-3\sqrt{2}\,\cotanh\left(\sqrt{\frac{-3\,E}{2}}\,s\right)+3\left(\cotanh\left(\sqrt{\frac{-3\,E}{2}}\,s\right)\right)^2
\right\rgroup\,;\\
\varphi_2(s) &=& \exp\left(-\sqrt{-3\,E}\,s\right)\left\lgroup
1+3\sqrt{2}\,\cotanh\left(\sqrt{\frac{-3\,E}{2}}\,s\right)+3\left(\cotanh\left(\sqrt{\frac{-3\,E}{2}}\,s\right)\right)^2
\right\rgroup
\,.\rule{0pt}{20pt}
\end{array}
\right.
\]
The curve is now completely described, since $(x',y')$ is obtained from $(\varphi_1,\varphi_2)$ by an affine transformation. The determinant of this transformation should be chosen so as to obtain $x'\,y''-x''\,y'=1$ (see also Figure~\ref{fig:III_II}, left).

In case (c) the function cotanh should be replaced by tanh in the above formulae. It should be noticed that the curvature function, as well as its derivative, tends towards the zero function for $E$ approaching zero. This is reflected in the accompanying Figure~\ref{fig:III_II} (middle). Furthermore, the parameter $E$ is merely a rescaling parameter, and all curves can be obtained from each other from  a (non-unimodular) affine transformation which preserves the depicted osculating parabola.

%%drieb_v1.pdf
%%II_v2.ps
%%III_v2.ps
%%
\begin{figure}
\begin{center}
\includegraphics[angle=0,width=0.32\textwidth]{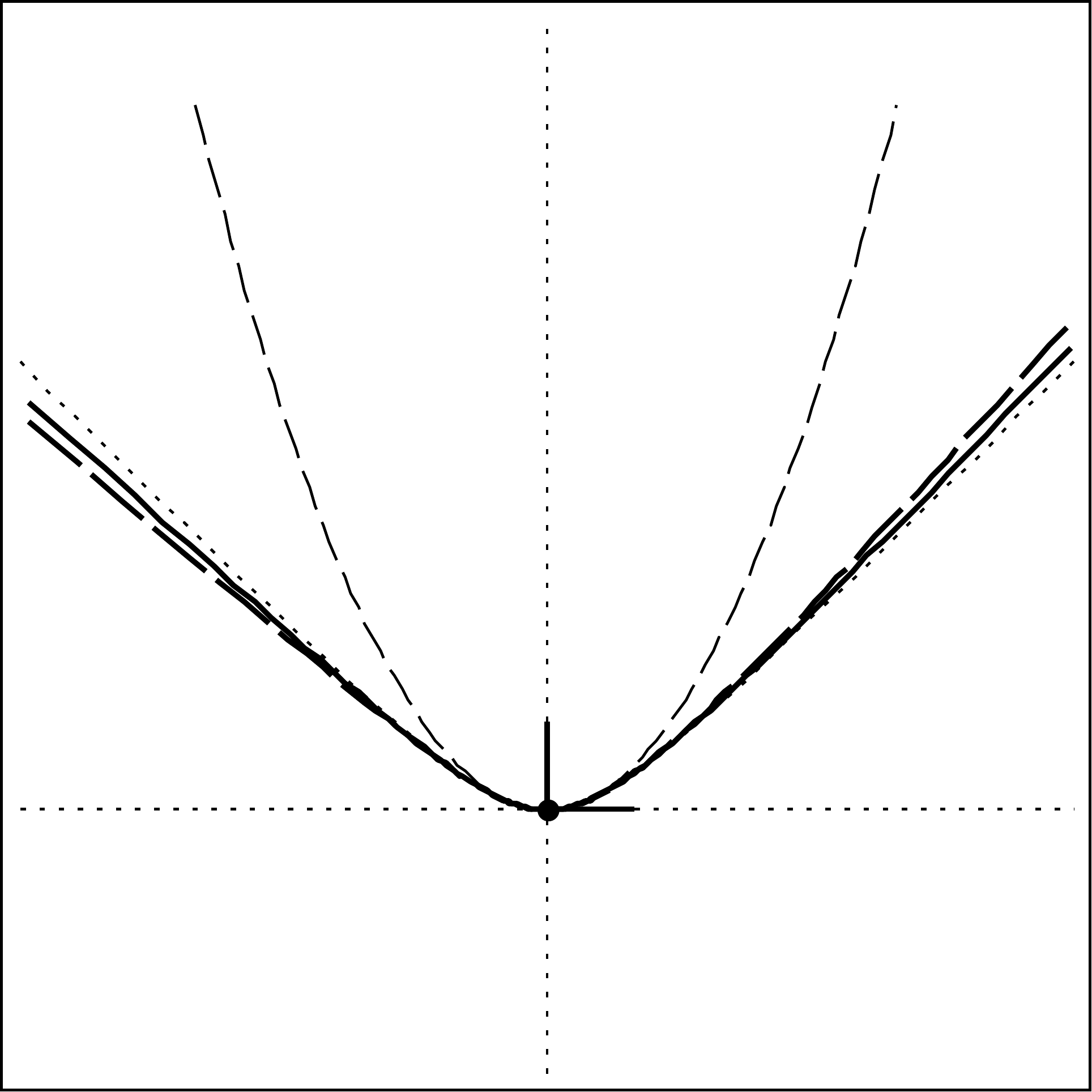}
\includegraphics[angle=0,width=0.32\textwidth]{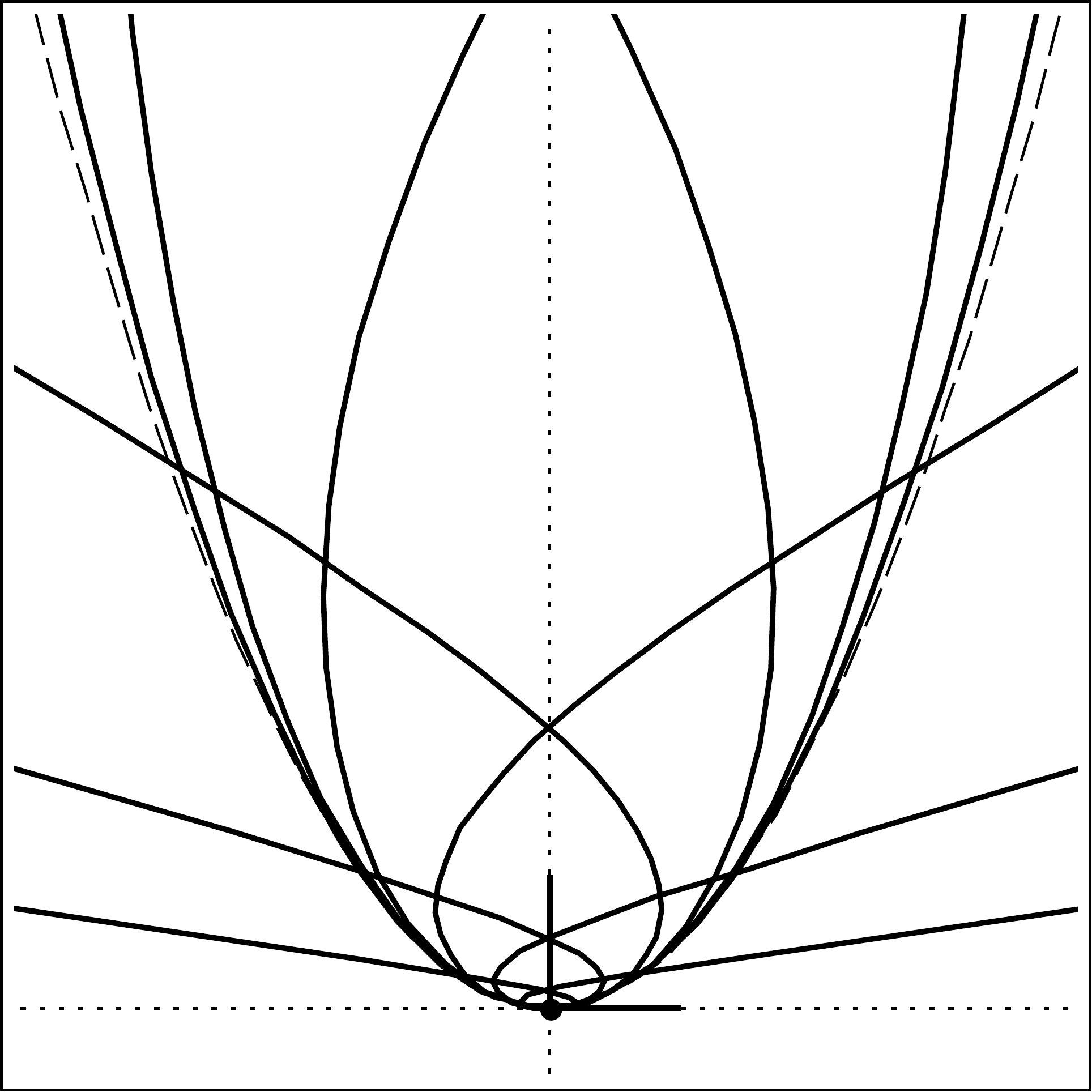}
\includegraphics[angle=0,width=0.32\textwidth]{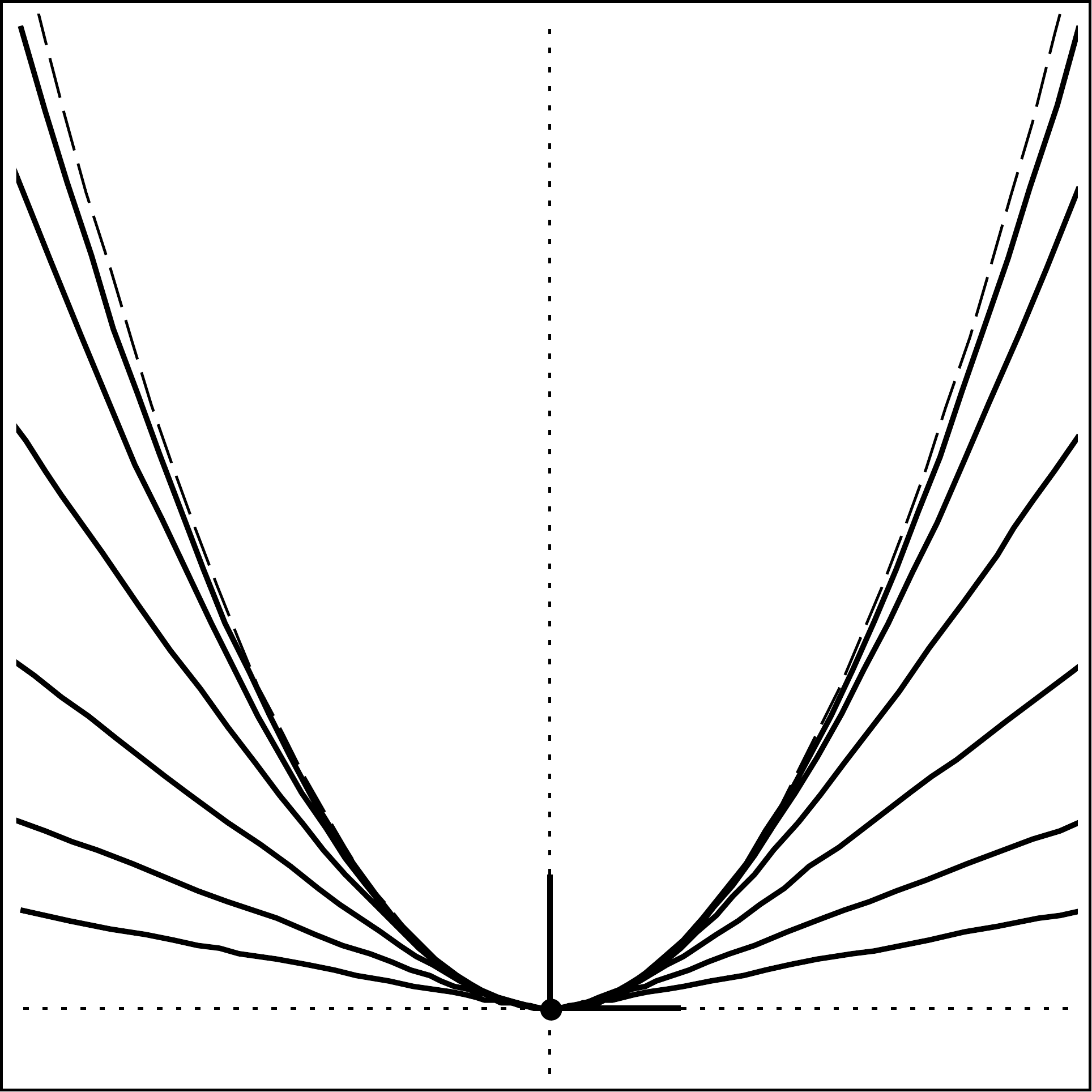}
\caption{In the left picture two curves of case (D.a), corresponding to $E=-1/4$ (curve in long dashed line) and $E=-1/2$ (solid line), are displayed. 
The curves have been drawn in such a way that they share their equi-affine Frenet frame, osculating parabola (dashed), and osculating hyperbola (dotted) at the point where their curvature takes the value $-1$.
In the middle figure, we have plotted the curves which are described in 
Case (D.c), where $E$ runs over the values $4^{n}$ for 
$n=-4,-3\ldots1$. These curves are displayed in such a way that they share their osculating parabola (dashed) and equi-affine Frenet frame (displayed as two small lines) at the parameter value $s=0$ where $\kappa$ achieves its maximal value. As $E$ becomes smaller, the curve more closely resembles this parabola. We notice that the curve has a point of intersection which tends to infinity for $E$ approaching zero. The right picture shows the curves of case (E), where $E$ runs over $-4^{n}$ for the same values of $n$ as before.
}
\label{fig:III_II}
\end{center}
\end{figure}

\noindent\framebox{\textsc{Case} (E). 
%%my notes: Case 3
$g_3 > 0$ 
and $\Delta= 0$.}\raisebox{-9pt}{\rule{0pt}{24pt}}\ \ 
Define $E=\sqrt[3]{g_3}>0$, then either $\kappa=3\,E$, which gives an ellipse, or
\[
\kappa(s) = -9\,E\,\left(\tan\left(\sqrt{\frac{3\,E}{2}}\,s\right)\right)^2-6\,E\,.
\]
Similar to the previous case, the curve can be retrieved since $x'$ and $y'$ are linear combinations of the following two functions, and is represented in Figure~\ref{fig:III_II}, right:
\[
\hspace{-5mm}
\left\{
\begin{array}{rcl}
\varphi_1(s) &=& \cos(\sqrt{3\,E}\,s) - 3\,\cos(\sqrt{3\,E}\,s )\left(\tan\left(\sqrt{\frac{3\,E}{2}}\,s\right)\right)^2
+3\sqrt{2}\,\sin(\sqrt{3\,E}\,s)\tan\left(\sqrt{\frac{3\,E}{2}}\,s\right)\,;\\
\varphi_2(s) &=& -\sin(\sqrt{3\,E}\,s) + 3\,\sin(\sqrt{3\,E}\,s )\left(\tan\left(\sqrt{\frac{3\,E}{2}}\,s\right)\right)^2
+3\sqrt{2}\,\cos(\sqrt{3\,E}\,s)\tan\left(\sqrt{\frac{3\,E}{2}}\,s\right)\,.\rule{0pt}{20pt}
\end{array}
\right.
\]

\begin{figure}
\begin{center}
\includegraphics[bb=180 10 430 380, width=0.3\textwidth]{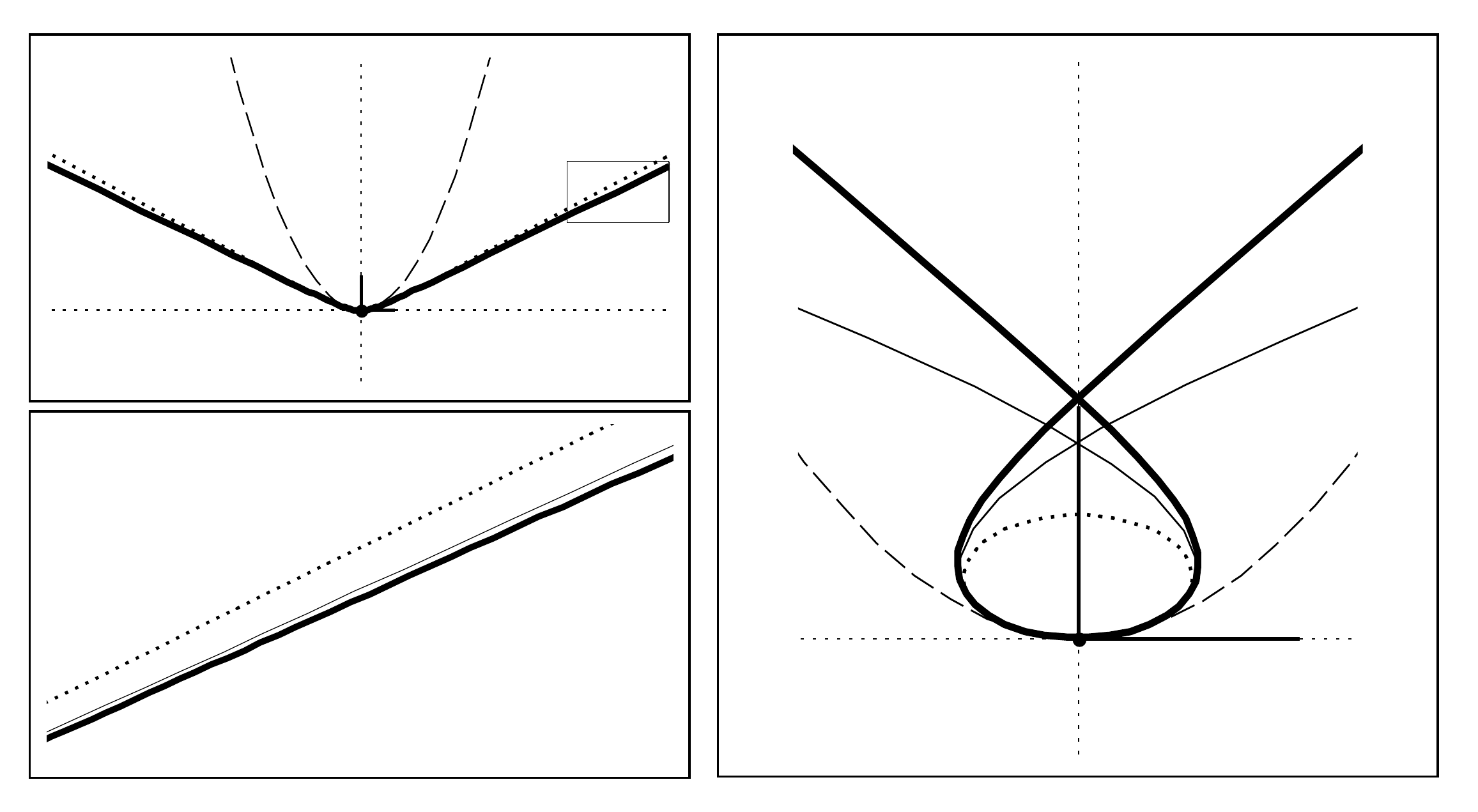}
\end{center}
\caption{Case (F). The top figure left displays such a curve for $g_3=1$, as well as its Frenet frame at the point of maximal curvature, and the osculating parabola and conic. The thin full line represents the curve of Case (D), whose parameter $E$ is adjusted in this way to have the same maximal curvature as $\gamma$. At the scale of the picture, the difference between these two curves is indiscernible, and for this reason a close-up of the small parallelogram which is indicated has been provided. Right: A curve corresponding to $g_3=-1$ (thick, solid line). The difference beween this curve and the curve of Case (E) (displayed by a thin line) which has the same osculating conic at $s=0$ and whose curvature has also a local maximum for $s=0$, is more easily recognised. The double-point of our curve curve does not coincide with the end-point of the Blaschke normal, but is emanated $1.0319$ times the Blaschke normal from the origin.}
\label{fig:bott}
\end{figure}
%%

%%%%%%%%%%%%%%%%%%%%%%%
\noindent\framebox{\textsc{Case} (\text{F}). $g_2= 0$ and $g_3 \neq 0$.}\raisebox{-9pt}{\rule{0pt}{24pt}}\ \ The constant $c_0$ which appears in the natural equation $\kappa(s)=-6\,\wp(s-c_0)$ can be taken equal to zero without restriction, because $\Delta<0$. The curve is described by a suitable affine transformation of $(\zeta(s),\wp(s)-(\zeta(s))^2)$.

Since the parameter $g_3$ merely rescales the curves under consideration, which additionally will turn out to be graphically comparable to the previous two cases (of Figure~\ref{fig:III_II}), we will merely display such a curve for $g_3=-1$ and $g_3=1$ (see Figure~\ref{fig:bott}). 

\noindent\framebox{\textsc{Case} (\text{G}). 
%%my notes: Case 1
$g_2=g_3=0$.}\raisebox{-9pt}{\rule{0pt}{24pt}}\ \ Apart from parabola's, we obtain the curves for which $(\kappa')^2/\kappa^{3}=2/3$,  and $\kappa(s)=-6/s^{2}$. These curves are equi-affinely congruent to $xy^4=1$.

This finishes our description of the critical points of $\int\kappa\,\dd s$ under area-contraint.

%%ae_4_v4.pdf
%%ae_6_v4.pdf
%%ae_7_v6.pdf.pdf
\begin{figure}
\begin{center}
\includegraphics[angle=0,width=0.32\textwidth,height=0.48\textwidth]{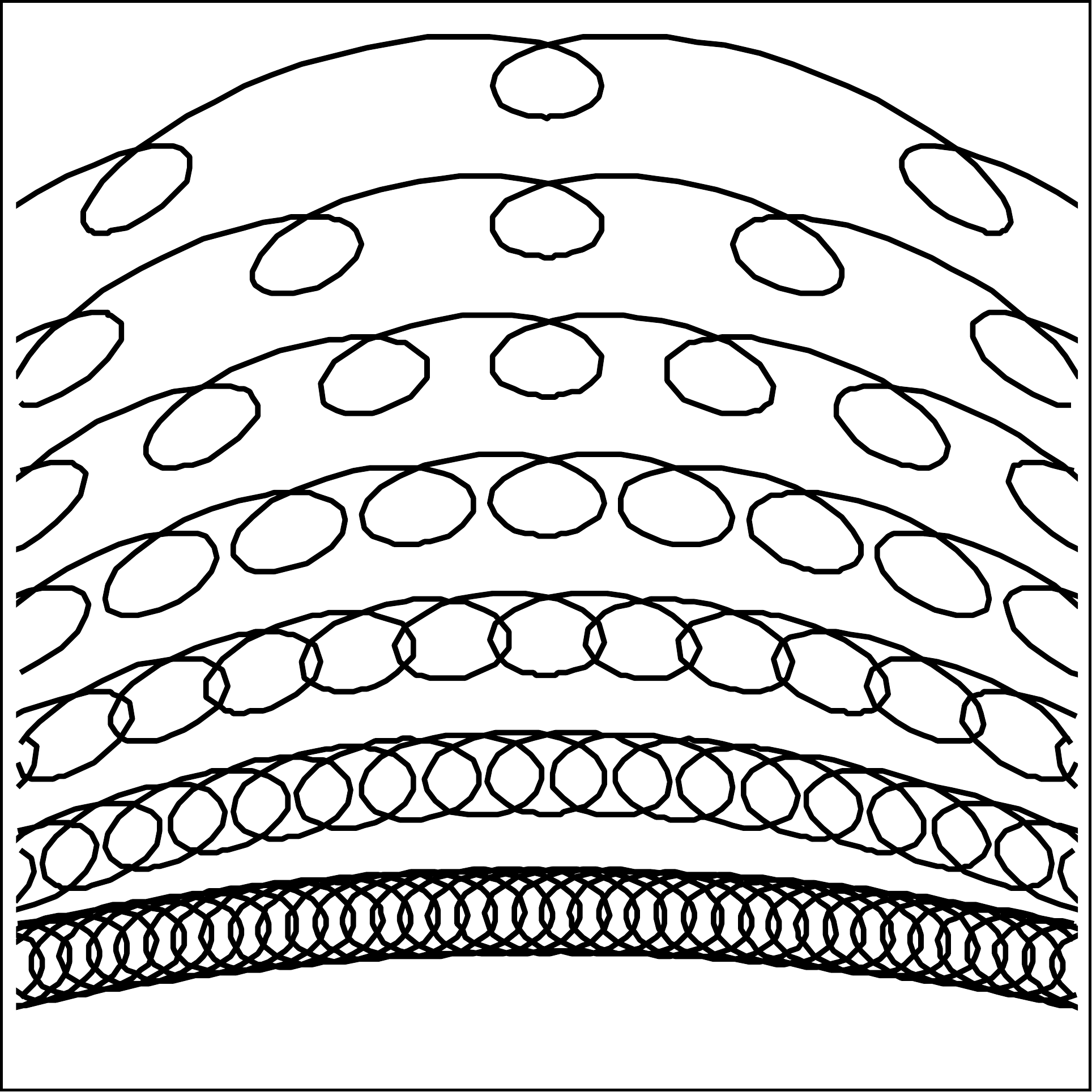}
\includegraphics[angle=0,width=0.32\textwidth,height=0.48\textwidth]{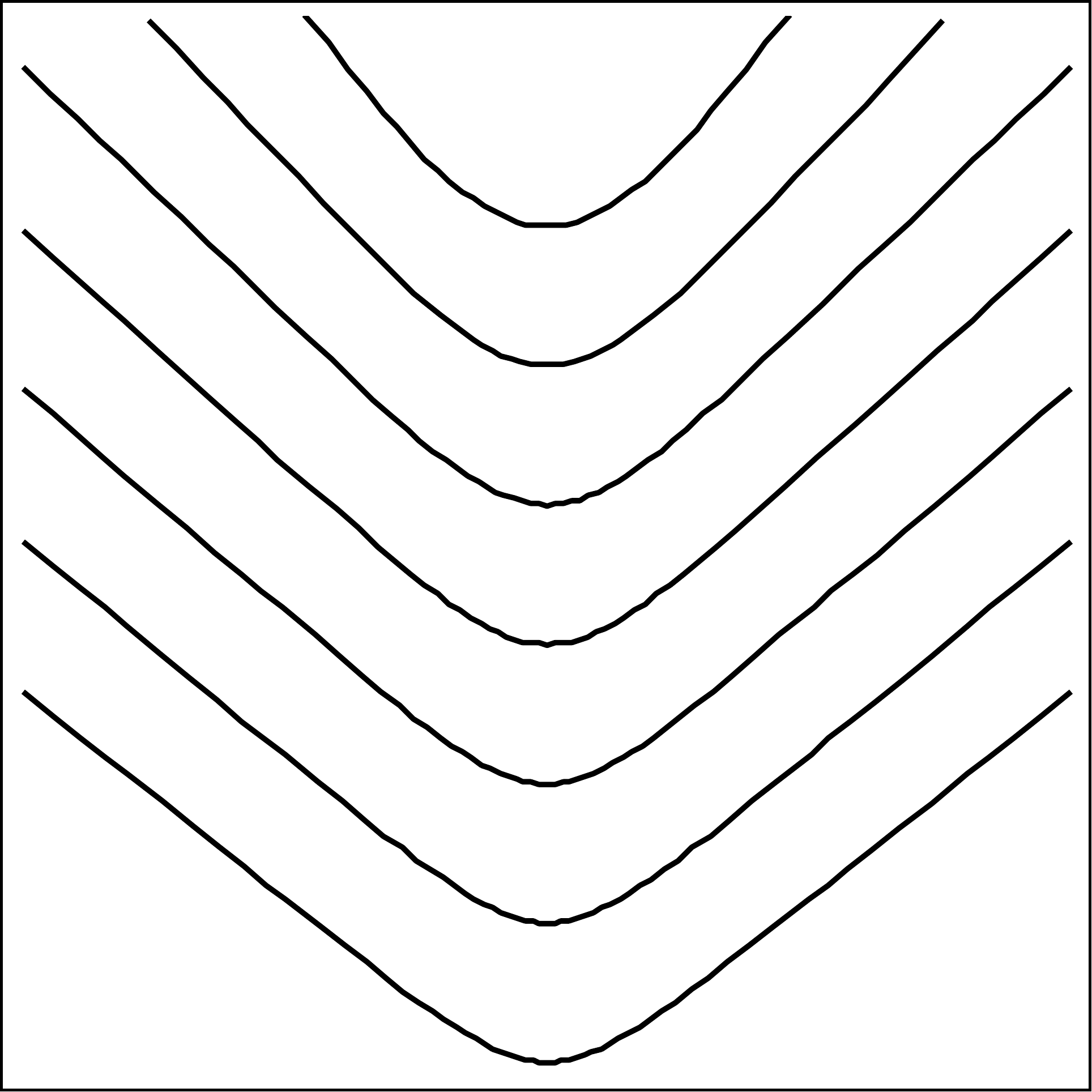}
\includegraphics[angle=0,width=0.32\textwidth,height=0.48\textwidth]{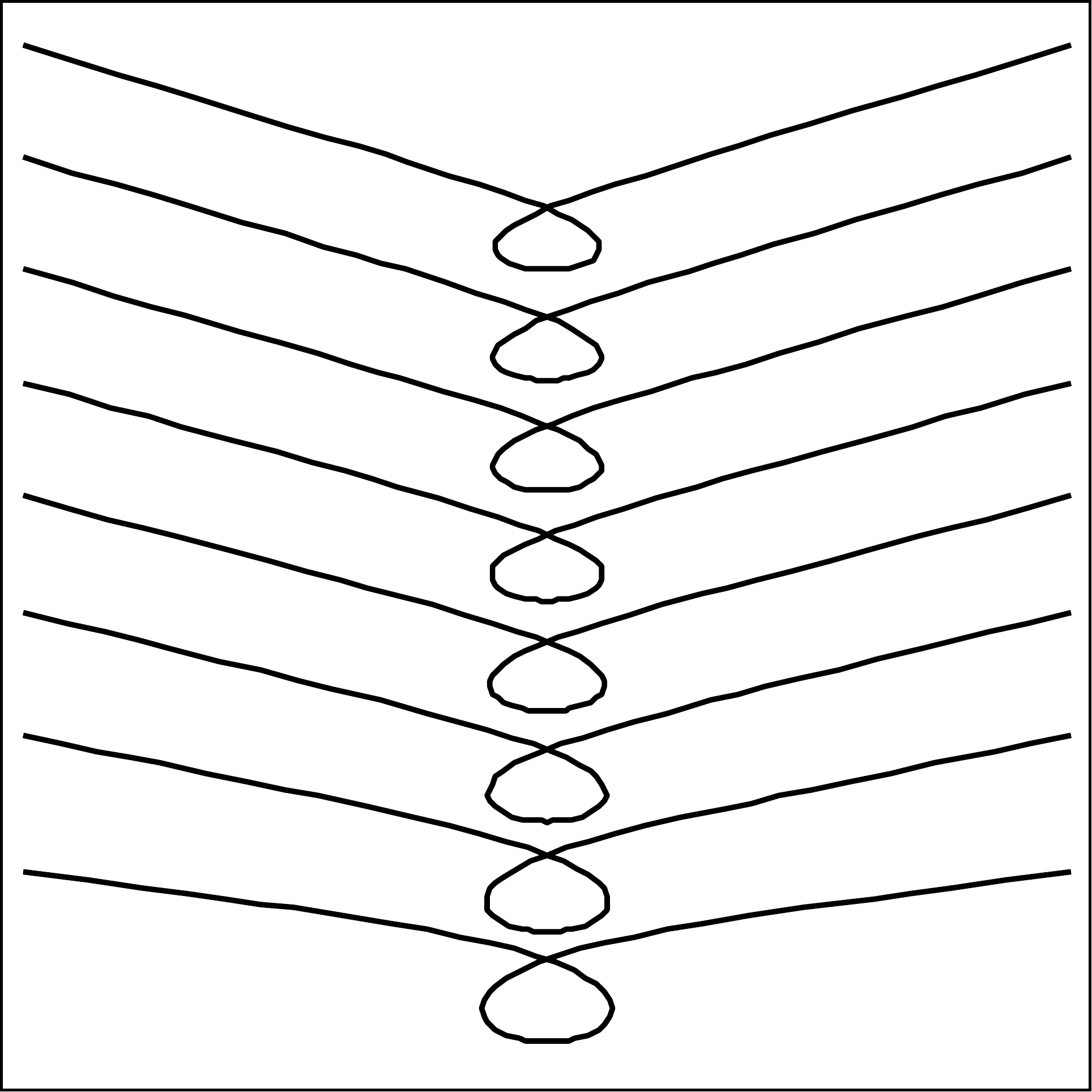}
\caption{Critical points of $\int \kappa\,\dd s$ under constrained $\int\!\dd s$: some examples with $A=1$. Left: from bottom to top, $B=-0.15,\ -0.10,\ldots,0.15$, with $c_0=\varpi_2$. Middle: similar, but $c_0=0$.  Right: $B=\frac{1}{6}+i$ ($i=1\ldots 8$, from bottom to top), and $c_0=0$.}
\label{fig:arclengthconstraint}
\end{center}
\end{figure}
%%

%%ae_8_v2.pdf
%%ae_9_v1.pdf
%%ae_10_v1.pdf
\begin{figure}
\begin{center}
\includegraphics[angle=0,width=0.14\textwidth,height=0.28\textwidth]{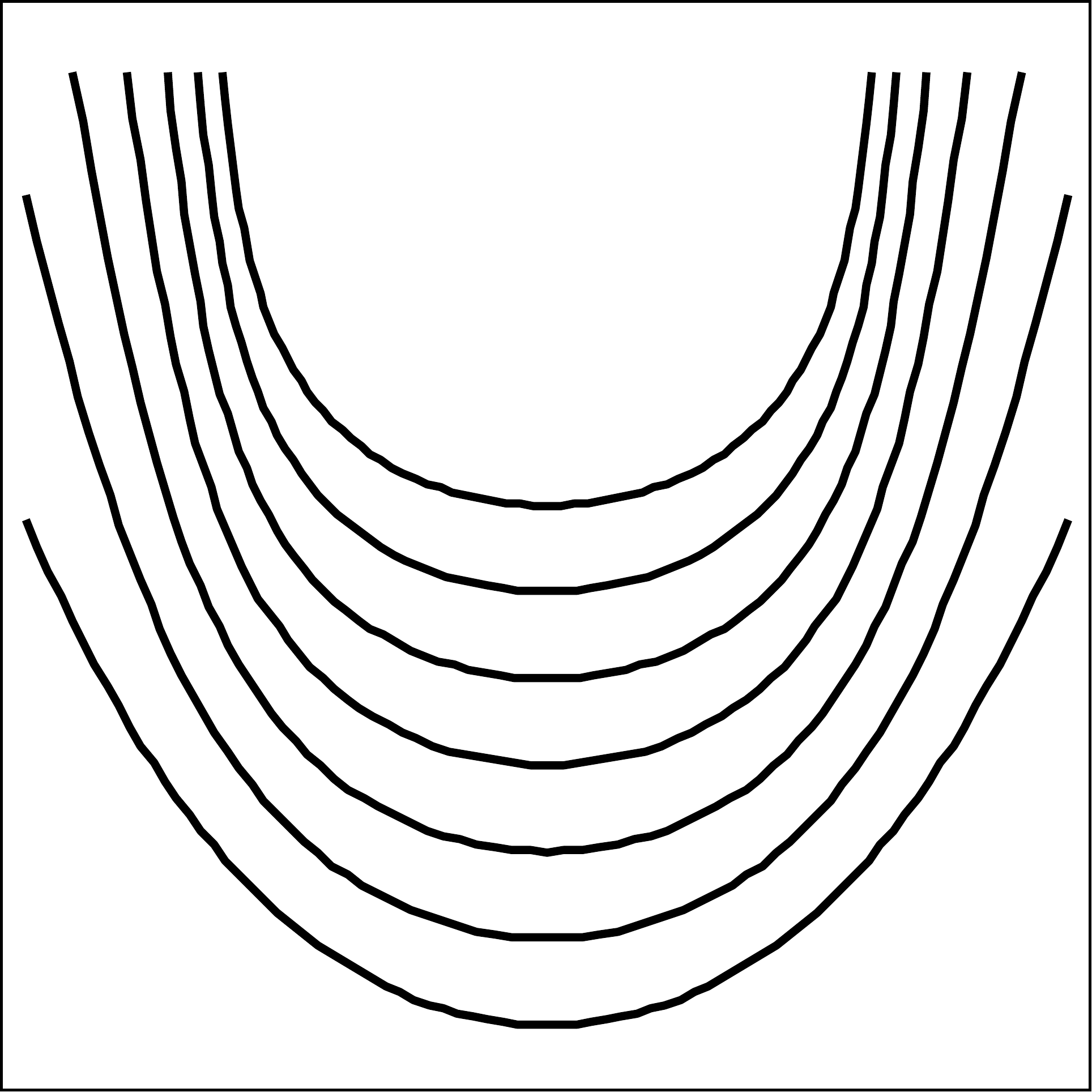}
\includegraphics[angle=0,width=0.56\textwidth,height=0.28\textwidth]{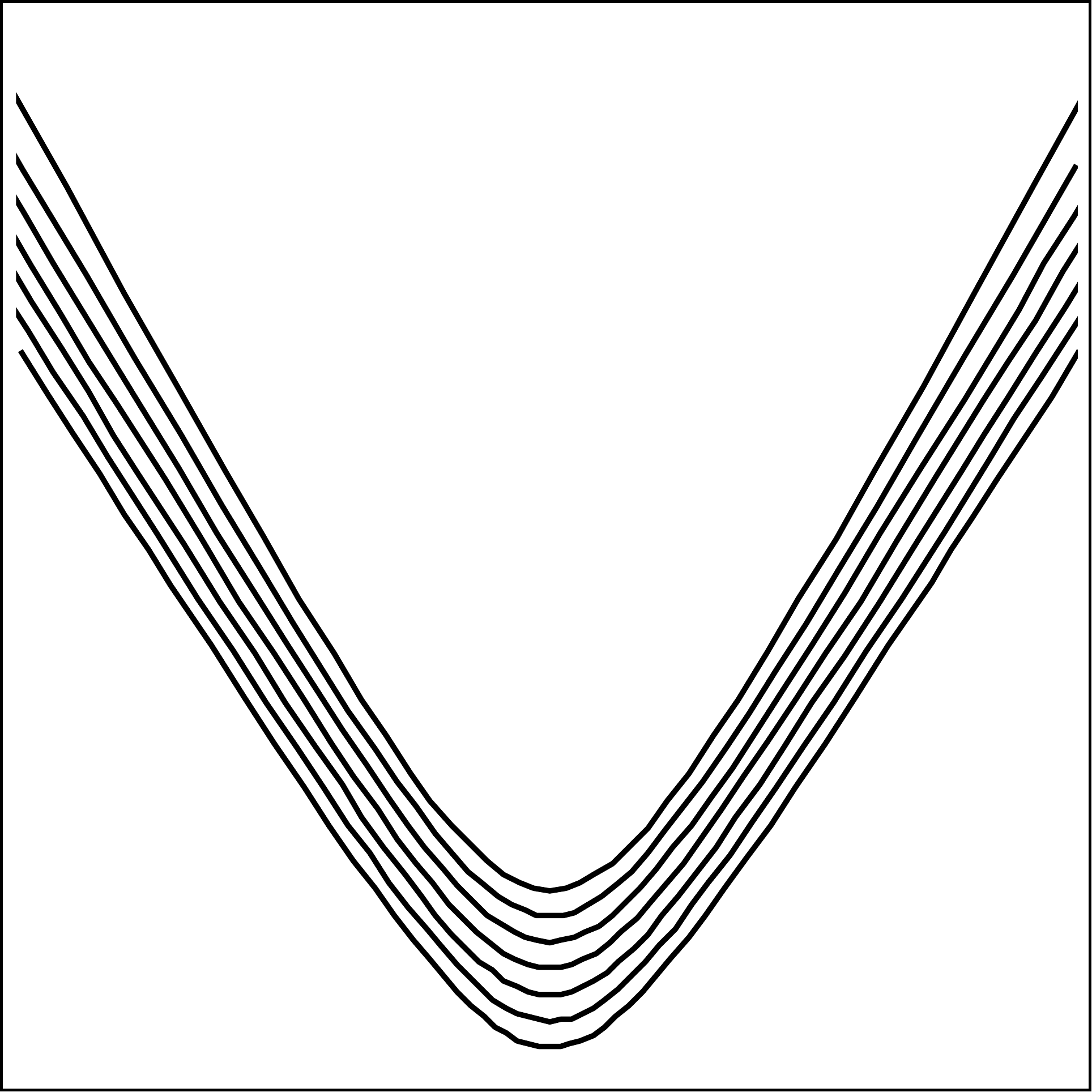}
\includegraphics[angle=0,width=0.28\textwidth,height=0.28\textwidth]{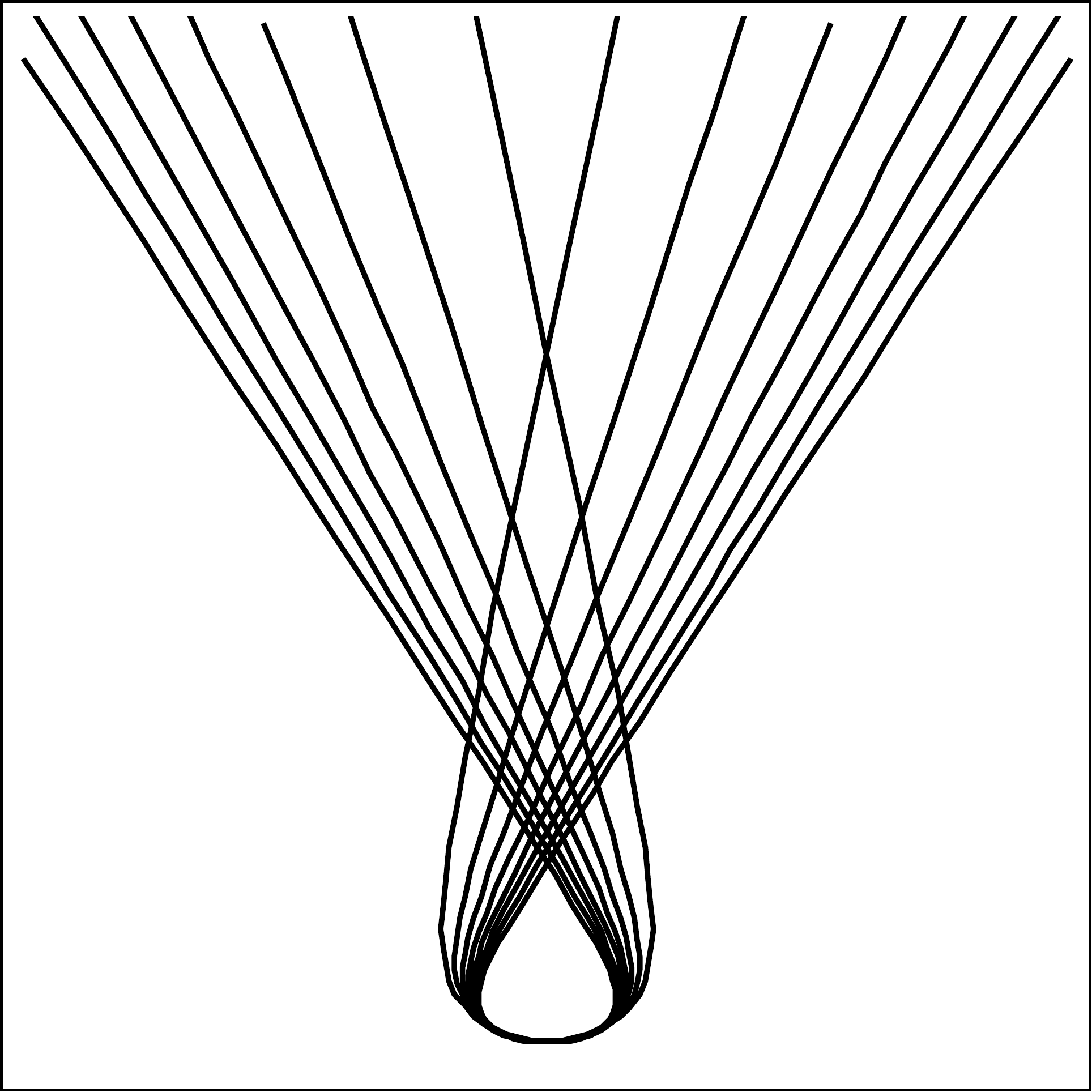}
\caption{Critical points of $\int \kappa\,\dd s$ under constrained $\int\!\dd s$: some examples with $A=-1$, for the same values of $B$ and $c_0$ as in the previous figure.}
\label{fig:arclengthconstraintcontd}
\end{center}
\end{figure}

\begin{remark}
It seems noteworthy to mention that the curves of case (F) have already been studied in another context. We recall that an affine-orthogonal net of curves is a pair of one-parameter families of curves for which at each point the tangent of the curve of the first family is the affine normal of the curve of the second family and conversely.
Now consider an affine-orthogonal net of curves which is invariant under the one-parameter family of translations generated by a vector $v$.
Then the curves for which the tangent is harmonically conjugate to $v$ w.r.t.\ the two tangents of the curves of the net precisely have natural equation 
$\kappa(s)=-6\,\wp(s)$ with $g_2=0$
(see \cite{bol1940}, pp.\ 147--152, and \cite{bottema1941}).
\end{remark}

\begin{remark}
The integration of the Euler--Lagrange equation~(\ref{eq:EL}), which characterises the critical points of the total equi-affine curvature 
w.r.t.\ deformations under which the total equi-affine arc-length is preserved, reveals that the function  $f=\frac{-\kappa}{6}+\frac{A}{12}$ satisfies the same equation $(f')^2=4\,f^3-g_2\,f-g_3$, where $g_2=\frac{A^2}{12}$ and $g_3$ stands for an integration constant. In the generic case when the corresponding cubic discriminant does not vanish, there holds $\kappa(s)=-6\,\wp(s-c_0)+\frac{A}{2}$, and the derivatives of the co-ordinate functions satisfy an adaption of the equation of Lam\'e, which can be integrated with the result
\[
\left\{
\begin{array}{rcl}
x(s)&=& -\zeta(s-c_0)+\frac{A}{12}\,s\,;\\
y(s)&=& \frac{A}{6}\,\zeta(s-c_0)\,(s-c_0)+\wp(s-c_0)
-\big(\zeta(s-c_0)\big)^2 - \left(\frac{A}{12}\,(s-c_0)\right)^2\,,
\end{array}
\right.
\]
up to an affine mapping. The constant $A$ can always be made 1, 0 or -1 by a suitable rescaling. 
We will not investigate this class of curves in detail but rather refer to Figures~\ref{fig:arclengthconstraint}--\ref{fig:arclengthconstraintcontd} where some instances of such curves are depicted.
\end{remark}

%%%%%%%%%%%%%%%%%%%%%%%%%%%%%%%%%%%%%%%%%%%%%%%%%%%%%%%%%%%%%%%%%%%%%%%%%%%
%%%%%%%%%%%%%%%%%%%%%%%%%%%%%%%%%%%%%%%%%%%%%%%%%%%%%%%%%%%%%%%%%%%%%%%%%%%

\section{The Functional $\displaystyle\int\sqrt{\kappa}\,\dd s$.}
%%%%%%%%%%%%%%%
\label{sec:sqrtcurv}

In this section, we will restrict our attention to \textit{strictly positively curved} curves, \textit{i.e.}, curves in the equi-affine plane with strictly positive equi-affine curvature. A functional which is of interest for such curves is given by $\int\sqrt{\kappa}\,\dd s$. 

A first motivation for the study of this functional is perhaps the fact that its Euclidean counterpart $\int \sqrt{\kappa_{\EE}} \,\dd s_{\EE}$ has been studied already in \cite{blaschke_vorlesungenI} \S\,27; in particular, the catenaries turned out to be the critical points of this functional.

But more importantly, the functional 
$\int\sqrt{\kappa}\,\dd s$ is an invariant with respect to the \textit{full affine} group. In the differential geometry of curves w.r.t.\ the full affine group, a \textit{full-affine arc-length} element and a \textit{full-affine curvature} have been defined 
(\cite{calugareanu_gheorghiu1941}; \cite{mihailescu_1}; \cite{schirokow} \S\,10.\textsc{v}; \cite{tutaev}). These are related with the equi-affine invariants by the formulae
\begin{equation}
\label{eq:aff_vs_full}
\dd s_{\FF}=\sqrt{\kappa}\,\dd s
\qquad\textrm{and}\qquad
\kappa_{\FF}=\frac{\kappa'}{2\,\kappa^{3/2}}\,.
\end{equation}

The class of curves for which the full-affine curvature is constant,
which includes for instance the logarithmical spirals, has been described already by 
%%begin footnote
Blaschke.\!\!\footnote{
See \cite[\S\,10]{blaschke_vorlesungenII}, but also \cite{bottema1941}, p.~91, footnote 2; \cite{leichtweiss}, p.~52, note 1.3.1; \cite{mihailescu_1}, \S\,5; \cite{tutaev}, p.~17, ex.~5.1.}
%%end footnote 
These curves, which will be called \textit{full-affine W-curves}, are precisely the orbits of a point under a one-parameter family of full-affine transformations.

It may be of interest to describe this full-affine arc-length and curvature without relying on equi-affine notions.
This is discussed in the four paragraphs (A)---(D) 
below.\!\!\footnote{The first three geometrical interpretations agree more or less with \cite{calugareanu_gheorghiu1941,mihailescu_1}. See also \cite{tutaev}.}

\noindent
\textbf{(A).}---On a curve $\gamma$ in the affine plane a point $M_1$ can be thought to approach a given point $M_0$. By intersection of the tangent and the affine normal at $M_1$ with the tangent at $M_0$ we obtain two points $A$ en $B$. This is illustrated in the left picture of Figure~\ref{fig:fullaff}, and according to the formula displayed below this picture the full-affine arc-length between $M_0$ and $M_1$ is a measure for the deviation of ${AM_0}$ from $BA$. It is immediately clear that this quantity vanishes for parabolas.

\noindent
\textbf{(B).}---A curve $\gamma$ has two points in common with its tangent line at a point $M_0$, four with its osculating parabola, and five with its osculating conic. There also uniquely exists a cubic which has a 
double-point at $M_0$ such that one branch is tangent to the affine normal whereas the other branch has six points in common with the curve. 
In the picture in the middle of Figure~\ref{fig:fullaff} the curve $\gamma$ is represented by the thick solid curve, and through the point $M_0$ of this curve have been erected the affine normal and the tangent. The osculating parabola, which follows the curve quite closely, is also drawn (full line), as is the osculating ellipse (dotted line) with centre $Q$. The aforementioned cubic (dashed) determines a unique point $A$ of intersection with the osculating parabola, disregarding $M_0$. 
Then, bringing the line $QA$ to intersection with the tangent at $M_0$ we obtained three collinear points from which the full-affine curvature $\kappa_{\mathbf{F}}$ at $M_0$ can be determined.

\noindent
\textbf{(C).}---For a point $M_0$ on the curve $\gamma$, consider the point $M_0^{\ast}$ which is antipodal to $M_0$ on its osculating ellipse. While $M_0$ traverses the curve $\gamma$, the point $M_0^{\ast}$ will describe a certain curve $\gamma^{\ast}$. In the right picture in Fig.~\ref{fig:fullaff} have been drawn in a dashed line the ellipses which are osculating at four points of the curve $\gamma$. If the line which is parallel to the tangent of $\gamma^{\ast}$ at $M_0^{\ast}$ is drawn through $M_0$, the points $P$, $Q$ of intersection with the osculating ellipse and the osculating parabola can be found. Then the full-affine curvature of the curve at the point $M_0$ is characterised by the ratio $\frac{QP}{PM_0}$ as well.

\noindent
\textbf{(D).}---Let us also give a new interpretation of the full-affine arc-length.

First, by a \textit{pointed parabola} will be understood a parabola in $\mathbb{A}^2$ of which a special point has been singled out. Because every two such parabolas can be matched by means of a unique equi-affine orientation-preserving transformation bringing the special points into correspondence, the space of all pointed parabolas can be simply seen as the equi-affine group, once a standard pointed parabola has been singled out.

%%%begin figure : cubic mathcal{C}
\begin{figure}
\begin{center}
\includegraphics[bb=170 0 500 242,width=0.5\textwidth]{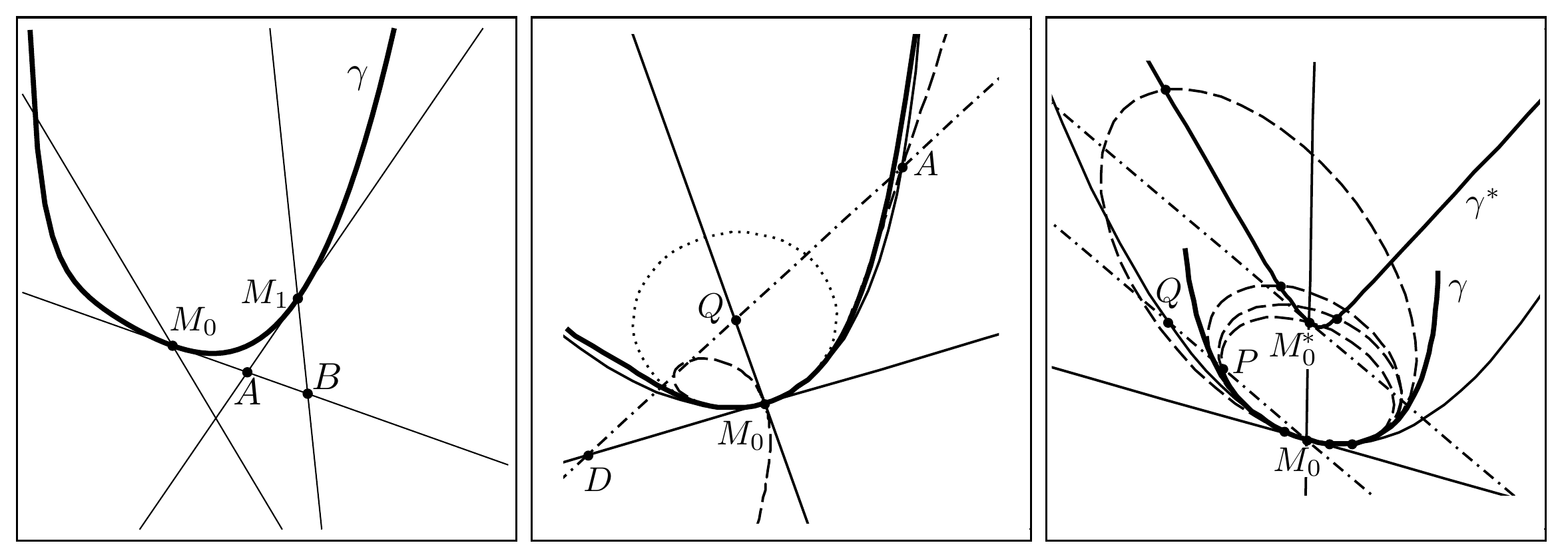}
\begin{tabular}{p{0.3\textwidth}p{0.3\textwidth}p{0.3\textwidth}}
\rule{0pt}{21pt}
$\big(\dd s_{\mathbf{F}}\big)_{01} \approx 
\sqrt{8\left(\frac{1}{2}-\frac{AM_0}{BM_0}\right)  }$.\rule{0pt}{21pt}
%1
& 
\rule{0pt}{21pt}
$\big(\kappa_{\mathbf{F}}\big)_{0}=\pm \sqrt{ \frac{25}{8}\,\frac{QD}{AD}}$. 
\rule{0pt}{21pt}
%2
& 
\rule{0pt}{21pt}
$\frac{QM_0}{PM_0} = 1 + 16\,\left[\big(\kappa_\textbf{F}\big)_0\right]^{2}$. 
\rule{0pt}{21pt}
%3
\end{tabular}
\caption{Interpretations of the full-affine arc-length and curvature.}
\label{fig:fullaff}
\end{center}
\end{figure}
%%%end figure : d

Next, we recall that a quadratic form $\varphi$ of index $q$  on $\mathbb{R}^{n+1}$ automatically endowes the hyperquadric $\varphi=-1$ with a pseudo-Riemannian metric of index $q-1$ by restriction to its tangent spaces. When this is applied to the Lie group 
$\mathrm{SL}(2)=\left\{
\left\lgroup
\begin{array}{cc}
a & b \\
c & d
\end{array}
\right\rgroup
\in\mathbb{R}^{2\times 2}\,\vert\,a\,d-\,b\,c=1\right\}$, where $\varphi=-\textrm{det}$ is regarded as a quadratic form, we conclude that the metric $g$ defined by
\[
g(v_P,v_P) = -\mathrm{det}(v_P)
\qquad\qquad\qquad
\textrm{(for $P\in\mathrm{SL}(2)$ and 
$v_P\in\mathrm{T}_P\mathrm{SL}(2)\subseteq\mathbb{R}^{2\times2}$),}
\]
which is obviously bi-invariant w.r.t.\ the group structure,
makes the special linear group into a pseudo-Riemannian space 
$(\mathrm{SL}(2),g)$ of Lorentzian signature which is isometric to the pseudohyperbolic space $H_1^3$ of constant sectional curvature $-1$.

Finally, for an arbitrarily parametrised curve $\gamma:\mathbb{R}\rightarrow \mathbb{A}^2: t \mapsto \gamma(t)$, the osculating parabola of $\gamma$ at $\gamma(t)$, endowed with the special point $\gamma(t)$, can be considered as an element of the equi-affine group. 
By omitting the translational part from this curve of osculating parabolas, a curve in a pseudo-Riemannian space results, which will be called the \textit{osculating parabolic congruence} of $\gamma$:
\[
\mathscr{P}_{\gamma} : \mathbb{R} \rightarrow (\mathrm{SL}(2),g)\,.
\]

We then have the following result, which is reminiscent to the interpretation of the Willmore integral of a surface in the conformal space as the area of its central sphere congruence (see \cite{blaschke_vorlesungenIII}, \S\,69 and 74.21):

\begin{theorem}
The full-affine arc-length of the curve $\gamma$ is precisely the pseudo-Riemannian arc-length of its osculating parabolic congruence $\mathscr{P}_{\gamma}$.
\end{theorem}

In order to illustrate the link between the pseudo-Riemannian geometry of 
$(\textrm{SL}(2),g)$ and the full-affine geometry of planar curves, we consider the following matrices which constitute a pseudo-orthonormal basis of $\mathrm{T}_{\textbf{1}}\mathrm{SL}(2)$:
\[
e_1 = 
\left\lgroup
\begin{array}{cc}
1 & 0 \\
0 & -1
\end{array}
\right\rgroup\,,
\quad
e_2 = 
\left\lgroup
\begin{array}{cc}
0 & 1 \\
1 & 0
\end{array}
\right\rgroup
\quad
\textrm{and}
\quad
e_3 = 
\left\lgroup
\begin{array}{cc}
0 & 1 \\
-1 & 0
\end{array}
\right\rgroup\,.
\]
The geodesics of $\mathrm{SL}(2)$ which start at the unit matrix in the directions $e_1$, $e_2$ or $e_3$ can be calculated by means of the matrix-exponential mapping, which is equal to the Riemann-exponential mapping because of the bi-invariance of the metric. These three curves can be considered as a one-parameter family of parabolas in $\mathbb{A}^2$, once a standard pointed parabola has been chosen.
In the three Figures~\ref{fig:geodpara} on p.\ \pageref{fig:geodpara}, this standard pointed parabola is displayed in thick, along with sixteen other parabolas which represent eight points on both sides of the corresponding geodesic in $\mathrm{SL}(2)$ at a distance of $\pi/20$ in-between. All drawn parabolas share their special point, which is situated in the middle of the figures.
As is clear from the left figure, not every curve in $\mathrm{SL}(2)$
can be realised as an osculating parabolic congruence from a curve in $\mathbb{A}^2$. For indeed the unit tangent vector field of such a hypothetical curve could only be horizontal, which would mean that the curve is a horizontal line, which does not admit osculating parabolas.
On the other hand, the curves of parabolas in the middle (resp.\ right) figure are the osculating parabolic congruence of a hyperbola (resp.\ an ellipse ).
The curve of parabolas in the right figure corresponds to a segment of a closed time-like geodesic of $\mathrm{SL}(2)$.

\begin{figure}
\begin{center}
\includegraphics[angle=0,width=0.32\textwidth]{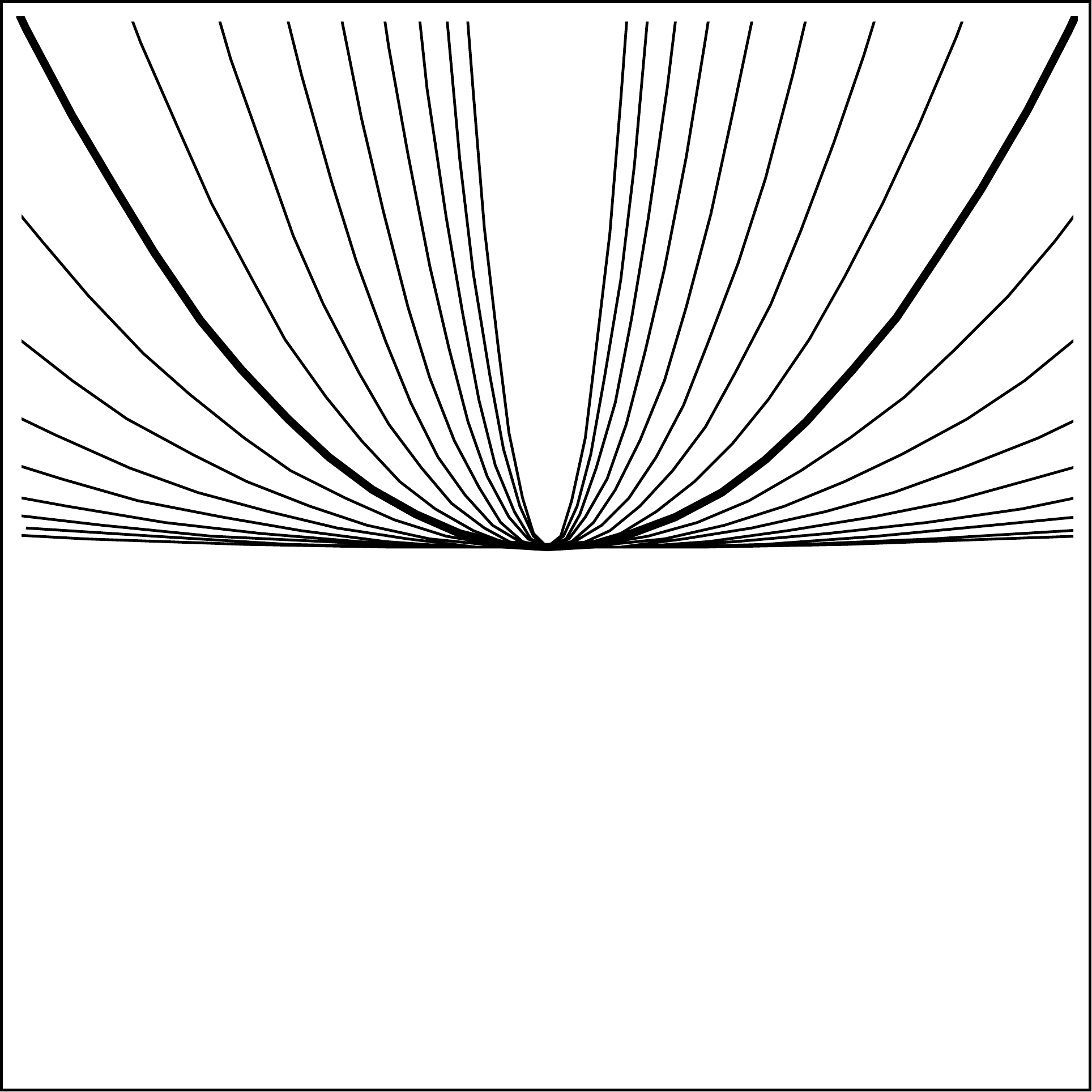}
\includegraphics[angle=0,width=0.32\textwidth]{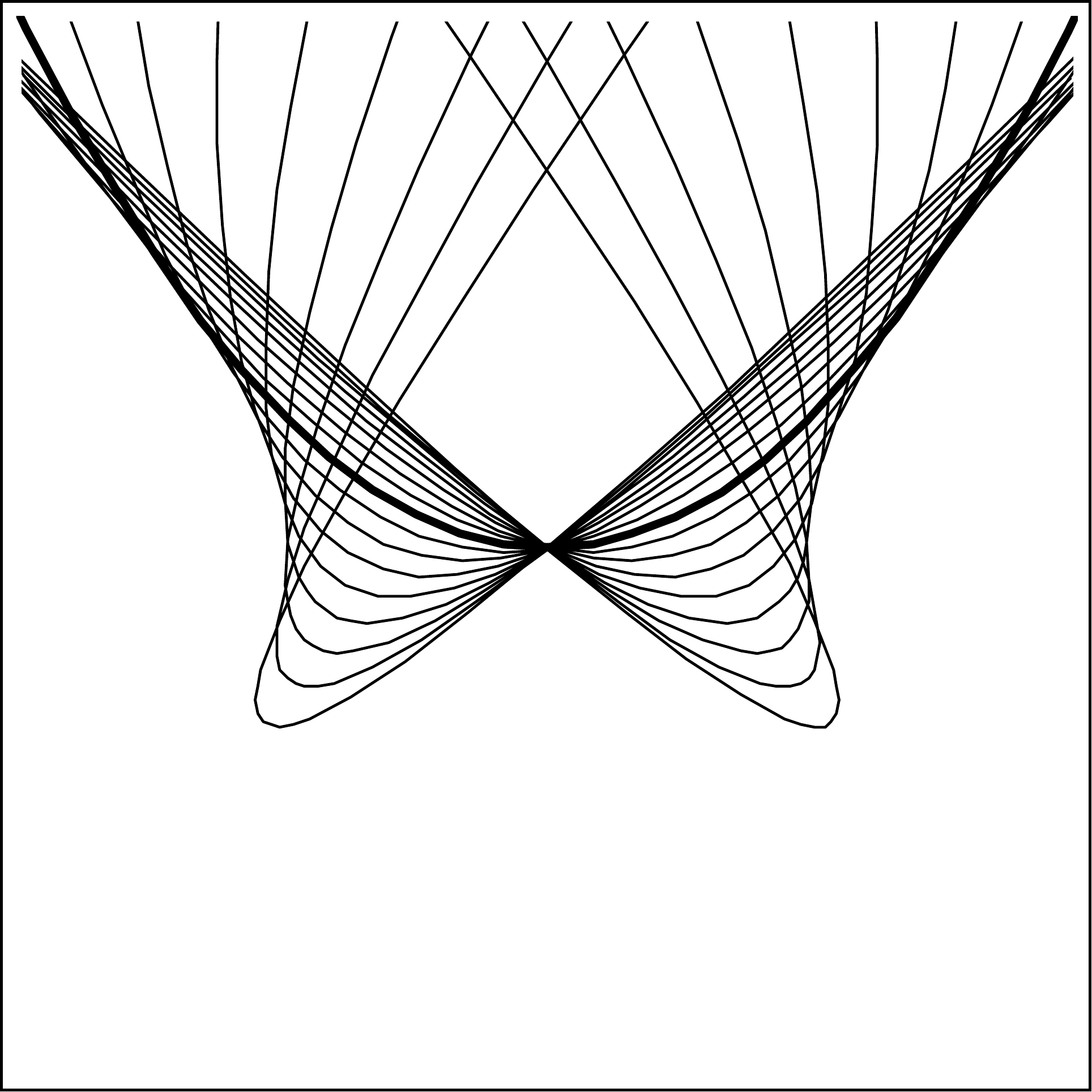}
\includegraphics[angle=0,width=0.32\textwidth]{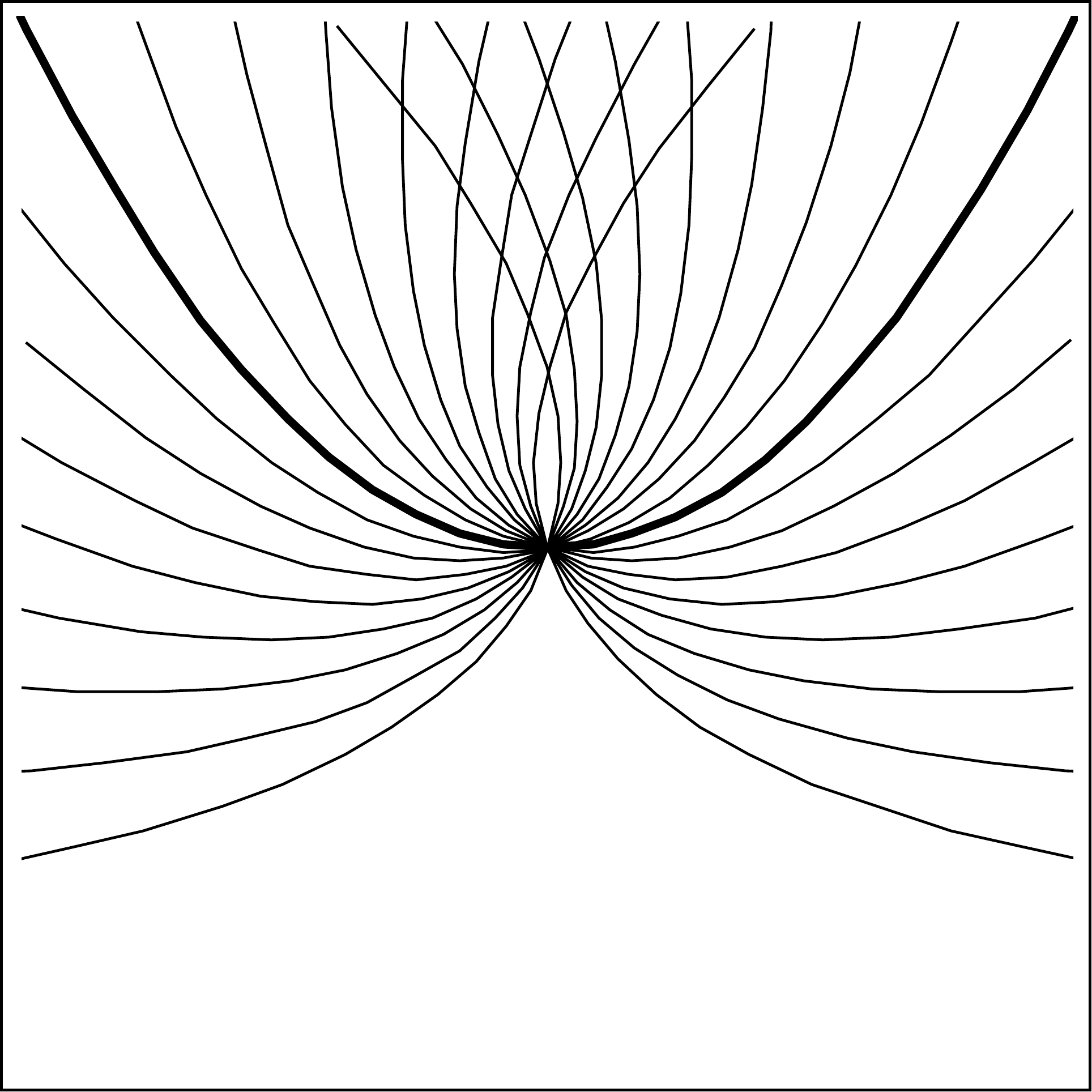}
\caption{Curves of parabolas corresponding to geodesics in $(\mathrm{SL}(2),g)$ emanating from the unit matrix in the directions $e_1$, $e_2$ and $e_3$, respectively.}
\label{fig:geodpara}
\end{center}
\end{figure}

\threeasterisks

Let us now return to the full-affine arc length from the variational point of view.
We remark that unlike the Euclidean and equi-affine variational formulae
\[
\left\{
\begin{array}{rclcl}
\delta \int \dd s_{\EE} &\propto& \int f\, \kappa_{\EE}\, \dd s_{\EE}
&\qquad& \textrm{(w.r.t. variational vector field $f\,N_{\EE} + g\,T_{\EE}$);}\\
\delta \int \dd s &\propto& \int f \,\kappa \,\dd s
&\qquad\rule{0pt}{14pt}& \textrm{(w.r.t. variational vector field $f\,N + g\,T$),}\\
\end{array}
\right.
\]
a similar formula does not hold for the full-affine geometry. In fact, the following variational formula w.r.t.\ the variational vector field $f\,N+g\,T$ can be deduced from (\ref{eq:delta_functional}):
\begin{equation}
\label{eq:deltasqrtkappa}
\delta\int\sqrt{\kappa}\,\dd s 
= 
-\int \frac{f}{12} \left\lgroup 
\left(\frac{\kappa'}{\kappa^{3/2}}\right)\!\!\rule{0pt}{14pt}'''\, 
+\kappa \left(\frac{\kappa'}{\kappa^{3/2}}\right)\!\!\rule{0pt}{14pt}'\,
\right\rgroup
\,\dd s\,.
\end{equation}

\begin{theorem}
\label{thm:intsqrt}
Among the strictly positively curved curves in the affine plane, the critical points of the full-affine arc-length are precisely the full-affine W-curves and the curves for which the full-affine curvature is a non-zero linear function of the position vector w.r.t.\ a suitable origin.
\end{theorem}
\begin{proof} 
In view of the variational formula (\ref{eq:deltasqrtkappa}), the equi-affine curvature function $\kappa$ and the full-affine curvature function $\kappa_{\FF}$ of such a curve are related by
\begin{equation}
\label{eq:kappakappaF}
(\kappa_{\FF})''' + \kappa \, (\kappa_{\FF})'=0\,.
\end{equation}
According to (\ref{eq:defkappa}), the functions $x$, $y$ and $1$ span the set of solutions of the homogeneous third-order differential equation $\xi'''+\kappa\,\xi'=0$, and therefore we have $\kappa_{\textbf{F}}= A\,x+B\,y+C$.
If $A$ or $B$ is not equal to zero, a suitable translation of the co-ordinate system reduces this equation to the form $\kappa_{\mathbf{F}}=A\,x+B\,y$.  
\end{proof}

\begin{theorem}
(i). The ellipses are the only simple closed strictly positively curved curves
which are a critical point of $\int \sqrt{\kappa}\,\dd s$.

\noindent
(ii). The ellipses are the only simple closed strictly positively curved curves which are a critical point of $\int \sqrt{\kappa}\,\dd s$ under area constraint.

\noindent
(iii). The ellipses are the only simple closed strictly positively curved curves which are a critical point of $\int \sqrt{\kappa}\,\dd s$ under equi-affine arc-length constraint.

\noindent
(iv). The ellipses are the only simple closed strictly positively curved curves which are a critical point of $\int \sqrt{\kappa}\,\dd s$ under total equi-affine curvature constraint.
\end{theorem}
\begin{proof}
(i).
Assume that a simple closed strictly positively curved curve is a critical point of $\int \sqrt{\kappa}\,\dd s$ but is not an ellipse.

Because the ellipses are the only closed strictly positively curved 
full-affine W-curves, there holds $\kappa_{\mathbf{F}}=A\,x+B\,y$, where at least one of the constants $A$ or $B$ is non-zero. The equation $A\,x+B\,y=0$ defines a straight line in the affine plane which cuts the curve in at most two points, and we conclude that there will be at most two critical points of the equi-affine curvature, \textit{i.e.}, \textit{sextactic points}, which contradicts an adaption of the four-vertex theorem (see \cite{blaschke_vorlesungenII}, \S\,19).

\noindent
(ii). Assume that a simple closed strictly positively curved curve is a critical point of $\int \sqrt{\kappa}\,\dd s$ under area constraint, but is not an ellipse. The curve $\gamma=(x,y)$ will be described  w.r.t. a co-ordinate system which makes the support function $\rho$ strictly positive.
The Euler-Lagrange equation expresses  that 
\begin{equation}
\label{eq:kappaQ}
(\kappa_{\FF})''' +\kappa\, (\kappa_{\FF})' = Q
\qquad\qquad \textrm{(for some constant $Q$),}
\end{equation}
where $\kappa_{\FF}=
\frac{\kappa'}{2\,\kappa^{3/2}}$ as in (\ref{eq:aff_vs_full}) and primes denote derivatives w.r.t.\ $s$. 
If $R$ is a local primitive of the equi-affine support function $\rho$ then $Q\,R$ is a solution of this linear inhomogeneous differential equation for $\kappa_{\FF}$, whereas the solution of the corresponding homogeneous differential equation was described in the proof of theorem \ref{thm:intsqrt}.

Therefore the full-affine curvature function $\kappa_{\FF}$ of the curve has to satisfy
\begin{equation}
\label{eq:kappaQR}
\kappa_{\FF}+ A\,x+B\,y+C = Q\,R 
\end{equation}
for some constants $A$, $B$ and $C$. The left-hand side of the above equation is the sum of four periodic functions, and hence the right-hand side should also be periodic. However $R'=\rho$ does not change sign, and consequently there necessarily holds $Q=0$, which means that we are in the above case (i).

\noindent
(iii). This is similar to the previous case. Instead of the formulae (\ref{eq:kappaQ}) and (\ref{eq:kappaQR}) one should use, respectively,
\[
(\kappa_{\FF})'''+\kappa\,(\kappa_{\FF})' = Q\,\kappa
\qquad
\textrm{and}
\qquad
\kappa_{\FF}+ A\,x+B\,y+C = Q\,s \,.
\]
(iv). Here one should use, respectively,
\[
(\kappa_{\FF})'''+\kappa\,(\kappa_{\FF})' = Q\,\left(\kappa'' + \kappa^2 \right)
\qquad
\textrm{and}
\qquad
\kappa_{\FF}+ A\,x+B\,y+C = Q\,K \,,
\]
where $K$ is a local primitive of $\kappa$ (and hence non-periodic).
\end{proof}

\begin{remark}
From (\ref{eq:intkappaints}), we immediately obtain a \textit{full-affine isoperimetric inequality}
\[
\int \dd s_{\mathbf{F}} \leqslant 2\,\pi
\]
for simple closed strictly positively curved curves, in which equality is attained exactly for  ellipses (cf.\ \cite{heil1967}, \S\,10). 
Furthermore, the total full-affine curvature vanishes for every closed curve:
\[
\int \kappa_{\mathbf{F}}\,\dd s_{\mathbf{F}} 
= \frac{1}{2}\int \frac{\kappa'}{\kappa} \dd\,s
= \frac{1}{2}\int \dd \log\,\kappa =  0.
\]
\end{remark}

\begin{remark}
The  Euler--Lagrange equation (\ref{eq:kappakappaF}) can be written completely in full-affine invariants: 
\begin{equation}
\label{eq:Miha}
\frac{\partial^3 \kappa_{\FF}  }{\partial (s_{\FF})^3}  
+
3\,\kappa_{\FF}
\frac{\partial^2 \kappa_{\FF}  }{\partial (s_{\FF})^2}  
+
\left(
\frac{\partial \kappa_{\FF}  }{\partial s_{\FF}}  
\right)^2
+
\left(
2\,(\kappa_{\FF})^2+1
\right)\,\frac{\partial \kappa_{\FF}  }{\partial s_{\FF}}  
=0 \,.
\end{equation}
This equation, which \textit{should} have been obtained in 
\cite[eq.\ (33)]{mihailescu_2}, has been contended to
lack the simplicity of the Euler--Lagrange equation for the critical points of the projective arc-length functional which was derived by \'E. Cartan (\cite{cartan_1927}, eq. (16)):
\[
\frac{\partial^3 \kappa_{\PP}  }{\partial (s_{\PP})^3}  
+ 8 \,\kappa_{\PP} \,
\frac{\partial \kappa_{\PP}  }{\partial s_{\PP}} 
=0\,,
\quad
\textrm{or equivalently (for some $C\in\mathbb{R}$),}
\quad
\frac{\partial^2 \kappa_{\PP}  }{\partial (s_{\PP})^2}  
+ 4 \,(\kappa_{\PP})^2
=C\,.
\]
In this article we have seen 
in (\ref{eq:odekappa}) as well as in (\ref{eq:kappakappaF}) 
an Euler--Lagrange equation which is similar to Cartan's equation.
\end{remark}

%%fullaffine_v8.pdf
\begin{figure}
\begin{center}
\includegraphics[angle=0,width=0.48\textwidth]{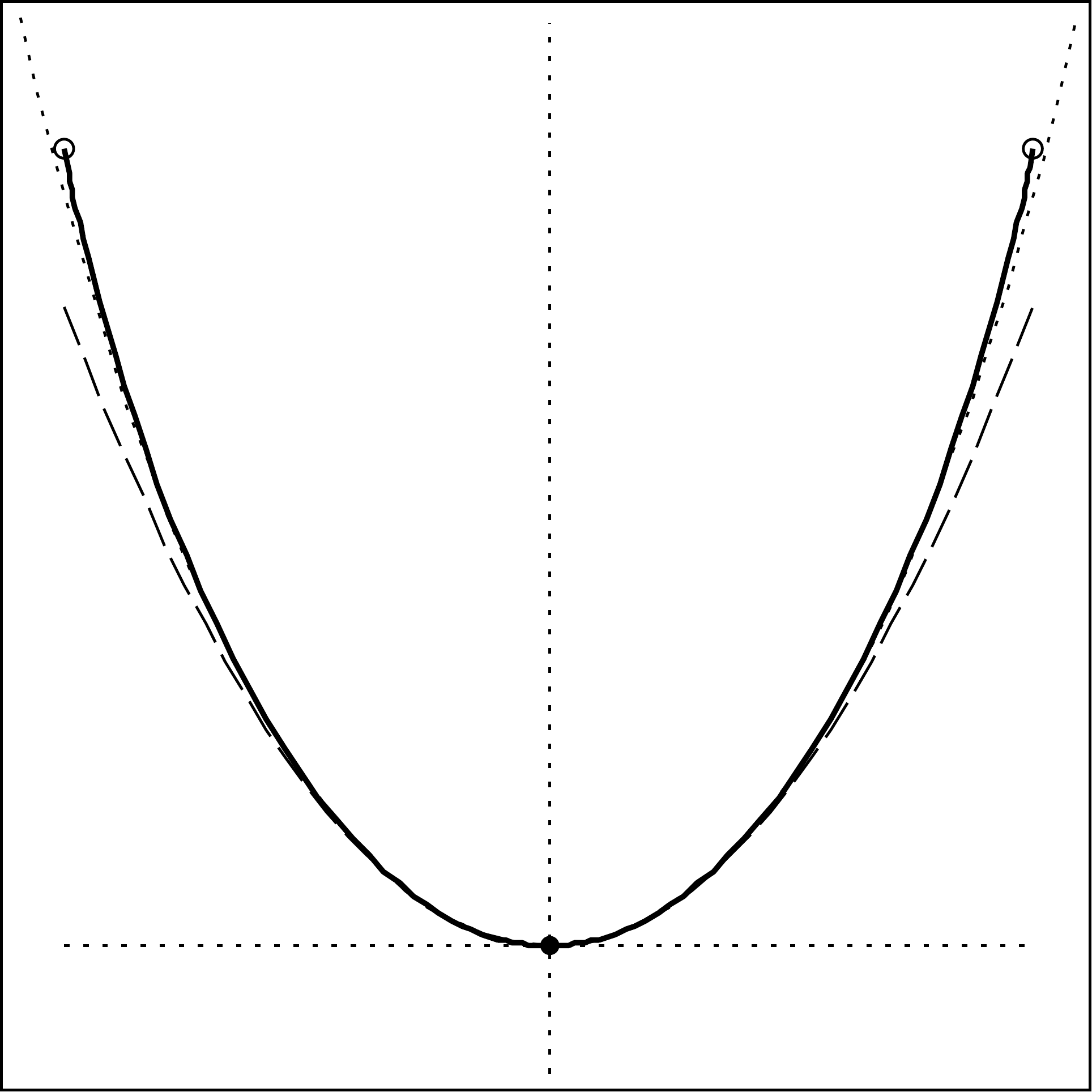}
\caption{The solid line represents the curve from Example~\ref{ex:fullaffineminimalcurve}. At the black dot, which represents the point where $\kappa_{\FF}$ vanishes, the tangent line and the affine normal line are shown, as well as the corresponding osculating parabola (dashed). An arc of the over-osculating ellipse at this point is drawn in a dotted line. As follows from the asymmetry of $\kappa_{\FF}$, the curve is symmetric around the point $s_{\FF}=0$. The point with parameter $s_{\FF}$ approaches one of the two end-points indicated by small circles when $s_{\FF}$ tends towards 
$\pm\infty$.
}
\label{fig:fullaffine}
\end{center}
\end{figure}

\begin{example}
\label{ex:fullaffineminimalcurve}
Let us end with an example: The curve given by $\kappa_{\FF}(s_\FF)=\frac{3}{\sqrt{2}}\mathrm{tanh}(\sqrt{2}\,s_{\FF})$ satisfies eq.~(\ref{eq:Miha}). This curve has been drawn in Figure~\ref{fig:fullaffine}. 
\end{example}

%%%%%%%%%%%%%%%%%%%%%%%%%%%

\noindent{\footnotesize\textsc{Previous Address:}
K.U.Leuven, Departement Wiskunde, Celestijnenlaan 200B, 3001 He\-verlee, Belgium.\\
\textsc{Previous Address:} 
Masaryk University, Department of Mathematics and Statistics, Kotl\'a\v{r}sk\'a 2, 611~37 Brno, Czech Republic.}

\end{document}